\documentclass[a4paper]{article}                     
\usepackage[left=2.5cm,right=5cm,top=2.5cm,bottom=2cm,includeheadfoot]{geometry}
\usepackage{float}
\usepackage{graphicx}
\usepackage{enumerate}
\usepackage{amssymb, amsmath, amsthm}
\usepackage{dsfont}
\usepackage{textcomp}
\usepackage{listings}
\usepackage{color} 
\usepackage{hyperref}
\usepackage{url}
\usepackage{tikz}
\usepackage{pgfplots}

\DeclareMathOperator*{\argmax}{arg\,max}

\usepackage{titlesec}
\titleformat{\section}[runin]{\normalfont\bfseries}{\thesection .}{0.5em}{}
\titleformat{\subsection}[runin]{\normalfont\bfseries}{\thesubsection .}{0.5em}{}

\newtheorem{proposition}{Proposition}[section] 
\newtheorem{defn}[proposition]{Definition}
\newtheorem{assumption}[proposition]{Assumption}
\newtheorem{thm}[proposition]{Theorem}
\newtheorem{lemma}[proposition]{Lemma}

\newtheorem{remark}[proposition]{Remark}

\usepackage[T1]{fontenc}
\usepackage{lmodern}

\usepackage{framed}

\newcommand{\N}{\ensuremath{{\mathbb N}}}

\newcommand{\R}{\ensuremath{{\mathbb R}}}

\newcommand{\Qs}{\ensuremath{Q}}

\newcommand{\E}{\ensuremath{{\mathbb E}}}
\newcommand{\Pro}{\ensuremath{{\,\mathbb P}}}

\newcommand{\yy}{\ensuremath{{y^*}}}

\def\SpecialChap#1{{\let\cleardoublepage\relax\chapter{#1}}}

\definecolor{ao(english)}{rgb}{0.0, 0.5, 0.0}

\definecolor{brown(web)}{rgb}{0.65, 0.16, 0.16}

	\title{Competition versus Cooperation: A class of solvable mean field impulse control problems}
\author{S\"oren Christensen\thanks{Christian-Albrechts-Universit\"at zu Kiel, Mathematisches Seminar,
		Ludewig-Meyn-Str. 4,
		24098 Kiel,
		Germany} \and Berenice Anne Neumann\thanks{University of Trier, Department IV, Universitätsring 19, 54296 Trier, Germany}\and Tobias Sohr\footnotemark[1]}

\begin{document}

\maketitle

\begin{abstract}
	We discuss a class of explicitly solvable mean field type control problems/mean field games with a clear economic interpretation. More precisely, we consider long term average impulse control problems with underlying general one-dimensional diffusion processes motivated by optimal harvesting problems in natural resource management. We extend the classical stochastic Faustmann models by allowing the prices to depend on the state of the market using a mean field structure. 
	In a competitive market model, we prove that, under natural conditions, there exists an equilibrium strategy of threshold-type and furthermore characterize the threshold explicitly. If the agents cooperate with each other, we are faced with the mean field type control problem. Using a Lagrange-type argument, we prove that the optimizer of this non-standard impulse control problem is of threshold-type as well and characterize the optimal threshold. Furthermore, we compare the solutions and illustrate the findings in an example. 	
\end{abstract}

\vspace{.8cm}


\section{Introduction}

Mean field game theory has been introduced by Lasry and Lions \cite{LasryJapanese2007} and by Huang, Malham\'{e}, Caines \cite{HuangNCE2006} to study Nash equilibria of differential games with many players, where each player controls a diffusion. The main feature of these games is that the players do not interact with the others individually, but only through the distribution of all players' states.  This gives rise to apply techniques similar to mean field approximation from physics to obtain approximate equilibria for $N$-player games. Applications include growth models, the production of an exhaustible resource by a continuum of agents, as well as opinion dynamics \cite{CainesHandbook, ParisPrinceton2010}. 
	In the classical diffusion model the equilibria of the limiting mean field game are given by a system of nonlinear partial differential equations with partly initial, partly terminal conditions, if one considers the analytic approach to mean field games, or they are given by a coupled forward backward stochastic differential equation, if one considers the probabilistic approach (for details consider \cite{BensoussanMFG,CarmonaMFG2018,CarmonaMFGPartTwo2018}).  Note, however, that also these systems are mostly intractable. Only in the case of linear dynamics and quadratic costs, explicit solutions have been obtained. More recently, several other types of mean field games have been introduced, such as finite state mean field games \cite{CecchinProbabilistic2018,GomesConti2013, NeumannComputation} or mean field games of stopping \cite{bertucci2018optimal,NutzMFStopping, NutzCompetitionRD, CarmonaBankRun}, in which case it is sometimes possible to obtain other explicitly solvable models which might yield a deeper understanding of the nature of mean field equilibria. First results on mean field games of impulse control for diffusion processes have been obtained in \cite{Bertucci}. The main focus there is on developing a general theory in a Brownian motion model including existence and uniqueness results of equilibria.
	
	Another branch of research covers mean field type control theory, which discusses a connected question. The difference is that in these models there is no competition between the agents, but they are assumed to cooperate. From a mathematical perspective, this means that there is only one decision maker who chooses a control to optimize the expected reward for the whole population.
	We refer to \cite{andersson2011maximum} for an early paper on the maximum principle for such problems and to \cite{Carmona2013} for a comparison of mean field games and mean field type control problems. Some real-world problems have reasonable interpretations both as a mean field game and a mean field type control problem, see, e.g., \cite{doi:10.1137/17M1119196} for an example inspired by pedestrian crowd dynamics.
	
	In this paper, we leave the classical setting of continuous stochastic control, but consider stochastic impulse control problems. Impulse control problems form the mathematical framework to study a continuous time model with interventions at adaptively chosen discrete time points only. 
	Such problems naturally arise whenever costs have to be paid for each interventions, e.g., in portfolio management with constant transaction costs (\cite{EasthamHastings1988}, \cite{Belak2019})  and control of the exchange rate by the central bank (\cite{MundacaOeksendal1998}). An overview on results for jump diffusions is given in \cite{OeksendalSulem2005}, see also \cite{korn1999} for a survey article with financial applications. Many of these articles are based on the fundamental connection between impulse control problems and quasi-variational inequalities developed originally in \cite{BLi}, see also \cite{C13impulse} for more references. This formulation is straightforward for discounted problems. Some more caution is needed for long-term average formulations which, however, play an important role as well, see \cite{doi:10.1137/16M1085991,helmes2017continuous,helmes2017weak,christensen2019nonparametric,christensen2019solution} for some recent references. 
	
	Here we consider a well-known impulse control problem that naturally arises in natural resource management. Its deterministic form is widely used to calculate optimal harvesting strategies and originates in the work of Martin Faustmann back in 1849. Over the last decades, different stochastic extensions of this classical model utilizing diffusions have been suggested and discussed, see \cite{clarke1989tree,willassen1998stochastic,gjolberg2002real,A04,alvarez2006does,shackleton2010harvesting} to name just a few. The underlying question may be formulated as follows:
	Assume that the volume of wood in a forest stand is modelled as a one dimensional diffusion process and that fixed non-zero costs occur for each time harvesting. When should the forest stand be optimally harvested and what is the optimal value?
	
	In this paper, we add as a new feature an interacting component of different agents by letting the reward of the agents depend on the average harvesting rate in the market. More precisely, we consider a mean field framework and assume that the reward depends on the expected harvesting rate, that in turn depends on a joint strategy all other agents are assumed to use. 	
	In this context two problems are of interest: First, we can assume that all agents compete with each other. 
	In this setting every agent wants to maximize his reward given the other agents' strategies and we are interested in finding equilibria of these game. Second, we can assume that all agents corporate with each other. In this setting they jointly choose one strategy that maximizes the reward of all agents together. The first problem is a mean field game, the second problem is a mean field control type problem.
	
	We prove, for a natural set of assumptions, that for both problems optimal strategies of threshold-type exist and characterize these thresholds explicitly. Moreover, we obtain uniqueness of the equilibrium in threshold strategies. Furthermore, we obtain that the threshold for the mean field game is smaller than any threshold for the mean field type control problem.
	As a second step we also consider a second price formation method, where the price depends on the overall supply of wood in the market. In this setting we can again prove the existence of optimal strategies of threshold type for both problems and obtain also a similar characterization result. However, here no uniqueness statement is possible and moreover, we obtain that the thresholds for the mean field game are larger than the thresholds for the mean field type control problem, so the relation is reversed. 

	The remainder of the article is structured as follows:
	In Section \ref{sec:model} we introduce the full model and the standing assumptions. In Section \ref{sec:preliminaries} we describe several preliminary results including the necessary results for the underlying impulse control problems, which are proven in the Appendix \ref{sec:imp_control_toolbox}. Section \ref{sec:game} presents the analysis of the mean field game, Section \ref{sec:problem} presents the treatment of the mean field type control problem. Section \ref{sec:comparison} then provides the comparison result regarding the thresholds for the game and the control problem. Section \ref{sec:EW} then describes the analysis of the second price formation mechanism. Finally, Section \ref{sec:examples} illustrates the theoretical findings in two examples where the dynamics of the forest stand are modeled by a logistic SDE.

\section{Model and Standing Assumptions}\label{sec:model}

In this section we introduce the model as well as the necessary assumptions used in this paper. In the first subsection we describe and motivate the general model.
In the second subsection we then formulate the standing assumptions.

\subsection{The Model}

To model the natural resource under consideration, e.g., in the original model the volume of wood in the forest stand, we use a stochastic process $(X_t)_{t\geq 0}$ that is a regular one-dimensional It\^o-diffusion on $\R_+= (0,\infty)$, whose dynamics is described by \[
dX_t=\mu(X_t)dt+\sigma(X_t)dW_t\] for a standard Brownian motion $W$ and continuous functions $\mu:\R_+\rightarrow \R$, $\sigma:\R_+\rightarrow \R_+$ that are sufficiently regular to guarantee a unique (strong) solution to the stated SDE. As usual, we denote by $\Pro_x$ probabilities for the process conditioned to start in the initial state $x$ and write $\E_x$ for the corresponding expectation operator.
We furthermore denote the speed measure by $M$ and the scale function by $S$. By assumption, their densities $m$ and $s$, resp., are given by (see e.g. \cite{borodin-salminen}) 
\begin{align*}
m(x)=\frac{2}{\sigma^2(x)}\exp\left({\int^{x}_a\frac{2\mu(y)}{\sigma^2(y)}dy} \right),\;\;s(x)=\exp \left({-\int^{x}_a\frac{2\mu(y)}{\sigma^2(y)}dy}\right)
\end{align*}
for some $a\in \R_+$.
The forest owner can at any time decide to cut the forest, in which case the process is restarted at an externally given level $y_0>0$. In contrast to other models, we follow the classical Faustmann literature and assume that $y_0$ cannot be chosen by the decision maker, which simplifies some expressions. Let us, however, remark that the basic ideas used in the following could be taken over to more general settings as well. Formally, the cutting decision is modelled by a sequences of stopping times $R=(\tau_n)_{n \in \mathbb{N}}$ satisfying $\lim_{n \rightarrow \infty} \tau_n = \infty$ a.s. 
We just consider strategies $R=(\tau_n)_{n \in \mathbb{N}}$ such that the controlled process fulfills $X^R_{\tau_n-} \ge y_0$ for all $n \in \mathbb{N}$, which we call \textit{admissible strategies} in the following.
Here, by $X^R$ we denote the solution to 
\[
X^R_t=X^R_0+\int\limits_0^t\mu(X^R_s)\ d s+\int\limits_0^ t\sigma(X^R_s)\ d W_s -\sum\limits_{n; \ \tau_n \leq t}(X^R_{\tau_n-}-y_0).
\]
The most interesting strategies in practice as well as in our following discussions are threshold strategies. These are strategies $R=(\tau_n)_{n \in \N}$ such that for a given threshold $y>y_0$ we have (with $ \tau_0=0$): \[
\tau_n=\inf \{t\geq \tau_{n-1}:X^R_t\geq y \},\;n \in \N.
\] We denote such a threshold strategy for a threshold $y>y_0$ by $R(y)$. 
In Figure \ref{fig:ssplot} we present a simulated sample path for such a strategy.

\begin{figure}
	\centering
	\includegraphics[width=0.9\linewidth]{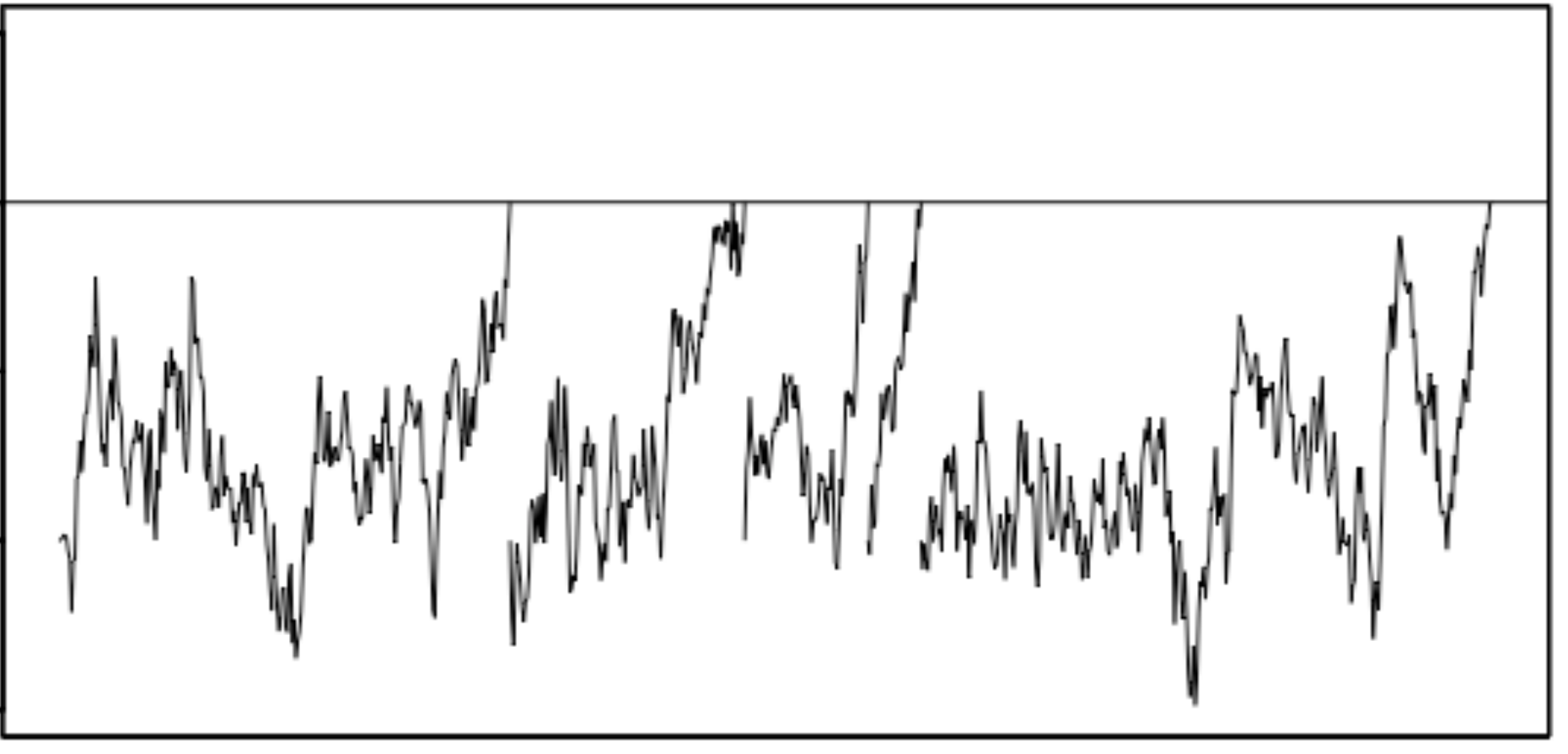}
	\caption{Simulated sample path with threshold strategy}
	\label{fig:ssplot}
\end{figure}

In the classical version of this famous impulse control problem, the forest owner chooses an admissible strategy in order to maximize his long term average reward 
\[
\liminf_{T\rightarrow\infty} \frac{1}{T}\E_x\left(\sum_{n:\tau_n\leq T}(\gamma(X_{\tau_n-}^R)-K)\right),
\] where $K>0$ denotes the fixed costs for each cutting decision and $\gamma(X_{\tau_n-}^R)$ can be interpreted as the income earned for selling the available wood at time $\tau_n-$.

However, this classical impulse control problem lacks one -- often crucial -- feature. It ignores that there might be other forest owners who also grow and sell their wood at the same time and therefore influence the prices.
Here we introduce these other agents in a mean field type fashion.
More precisely, we assume that there is a market for wood consisting of a continuum of agents with the same structure (dynamics and reward functional) as the decision maker under consideration. 
We remark that considering a continuum of agents is the natural way to formalize agents with negligible individual influence on the market outcome.

Since we consider a long-term average reward structure it is sensible to consider \textit{admissible stationary strategies} formally introduced in Definition \ref{def:admi_stat_strat} below.
Moreover, in this context it is natural to assume that the wood price depends on the average harvesting rate.
If all other agents use a joint strategy $Q=(\sigma_n)_{n \in \mathbb{N}}$, then a representative agent's forest stand is described by a process $\hat{X}^Q$ with the same dynamics as $X^Q$ and some initial distribution representing the collective distribution of the forest stands at time $0$. 
By standard renewal theory, the average collective harvesting rate is given by
\[
\frac{\mathbb{E}[X_{\tau-}^Q] -y_0}{\mathbb{E}[\tau]},
\]
where $\tau$ is the stopping time generating $Q$ (see Definition \ref{def:admi_stat_strat}). 

Now, to model our reward, we define for each pair of admissible strategies $R, \ \Qs$, with $\Qs$ stationary, the expected reward function \[
J_x(R,\Qs):= \liminf_{T\rightarrow\infty} \frac{1}{T}\E_x\left(\sum_{n:\tau_n\leq T}\left(\gamma \left(X_{\tau_n-}^R,\frac{\mathbb{E}[X_{\tau-}^Q] -y_0}{\mathbb{E}[\tau]}\right)-K\right)\right),
\]
with the interpretation that $J_x(R,\Qs)$ is the long term average reward obtained by an individual player starting in state $x$ playing according to strategy $R$, while the whole population plays according to strategy $\Qs$. Here, $\gamma(x,z)$ is the payoff function that models the reward the decision maker gets each time harvesting, which we assume in our model to depend on the average harvesting rate $\frac{\mathbb{E}[X_{\tau-}] -y_0}{\mathbb{E}[\tau]}$.

In this setting two optimization problems naturally arise. 
First, we could consider the behaviour of agents in a competitive market. 
In this setting all agents would -- given the choice of the others which they cannot influence -- maximize their reward. 
The central interest now lies in finding equilibrium strategies for this problem, that is finding strategies such that if the whole population of agents plays according to such an equilibrium strategy, then the individual agent has no incentive to deviate from doing so.
Formally, we search for admissible stationary strategies such that
$$J_x(\Qs,\Qs)  = \max_{R\mbox{ \footnotesize admissible}} J_x(R, \Qs)\mbox{ for all }x.$$

The second situation of interest is the case when all agents cooperate, that is all agents together choose a strategy  in order to maximize the overall reward. Mathematically speaking, we search for admissible stationary strategies such that
$$J_x(\Qs,\Qs) = \max_{R\mbox{ \footnotesize admissible}} J_x(R, R)\mbox{ for all }x.$$ 
The first problem is a classical problem in the theory of mean field games, the second is a classical problem in mean field type control theory, which both up to now have not been considered in detail in the context of impulse control problems. 
As in the standard theory, mean field equilibria are usually not solutions of the mean field type control problem and vice versa.
Indeed, it might be societal beneficial to coordinate on some strategy where unilateral deviations are still profitable and it might also happen that unilateral deviations are not profitable, whereas collective deviations are. 

\subsection{Assumptions}

In this section we introduce the standing Assumptions \ref{ass:process} and \ref{annahmenExistenzGG}, which we assume to hold for the rest of the paper, unless explicitly stated otherwise. 

\begin{assumption}\label{ass:process} 
	\begin{enumerate}
		\item The process $X$ is positively recurrent with integrable stationary distribution. The stationary distribution is, with a slight abuse of notation, denoted by $\Pro(X_\infty\in dx)$. Furthermore, we assume that $0$ is an entrance-boundary.
		\item There is a $y_1\geq y_0$ such that the drift function\ $\mu$ is strictly decreasing on $(y_1,\infty)$ and increasing on $(0,y_1]$.
		\item $\lim_{x\to\infty}s(x)=\infty.$
	\end{enumerate}

\end{assumption}

Let us briefly comment on the assumption regarding the process $X$:

1. In terms of the speed measure, the existence of a stationary distribution means that $M(\R_+)<\infty$. Then, up to standardization, $M$ is the stationary distribution of $X$
and we have \[\E [X_\infty ]=\int_{\R_+}x\frac{m(x)}{M(\R_+)}dx<\infty.\]
The assumption that $M(\R_+)<\infty$ implies that the boundary $\infty$ is natural and $0$ is either entrance or natural (see \cite[p.234]{KarlinTaylor}). In analytical terms, our additional assumption that 0 is not natural reads, for arbitrary $x>0$,  as
\[\int_0^x(S(x)-S(y))M(dy)<\infty.\]
Intuitively, it means that the process can reach an interior point from the state 0, but cannot reach the state 0 from an interior point.

2. The assumption on the drift $\mu$ intuitively means that if the quantity of available resources is low enough -- in our main example this means, young, new trees just got planted -- there is a random, yet on average positive growth of the drift. But at some point a level of saturation, for example due to the limited available space, is reached and there is no more room for further growth. Thus the average growth rate starts shrinking. 
This assumption not only seems practical on the intuitive level, it also includes the major growth models used for modelling natural resources. 
So it is satisfied the dynamics of the non-random Richards curve, which is a standard deterministic model for plant growth (\cite{Richardsbaum59,Pretsch19}) and also one of the most widely used models to describe biologic growth in a random environment, namely the generalized logistic (Perlhurst-Vearl-)diffusion satisfies this assumption.
This model is given by the SDE $$
dX_t=aX_t(1-bX_t)dt+\sigma(X_t)dW_t,  $$
where $a,b>0$ and $\sigma$ is a positive function (when $\sigma$ is linear, this becomes the (standard) logistic diffusion) and it will be  analysed in our example in Section \ref{sec:examples}.  

3. The assumption on the speed density $s$ could be relaxed, see the proofs below, but holds in the practically relevant examples. 

\begin{remark}\label{rem:zeta}
Assumption \ref{ass:process}.1 means that for all $x,y \in  \R_+$ we have $\E_x( \mathring\tau_y)<\infty$ for the hitting time
$ \mathring\tau_y=\inf\{t\geq0:X_t=y\},$
 see \cite[II.12]{borodin-salminen}. Hence we can define the function \begin{align}\label{eq:xi}
\xi:  \R_+ \rightarrow [0,\infty); \ y \mapsto \E_{y_0}(\tau_y)
\end{align}
for the threshold time $\tau_y=\inf\{t\geq0:X_t\geq y\}.$
We then obtain $\xi(x)$ for $x\geq y_0$ as 
\begin{align*}
\xi(x)&=  \int_{y_0}^x(S(x)-S(y))m(y)\ dy+(S(x)-S(y_0))M[0,y_0]
\end{align*} 
see \cite{helmes2017continuous}. 
In particular,  $\xi\in C^2$ on $[y_0,\infty)$. 
\end{remark}

As a next step let us explicitly define the considered strategies:

\begin{defn}\label{def:admi_stat_strat}
	An impulse control strategy $R=(\tau_n)_{n\in\N}$  is called admissible stationary strategy if there exists a stopping time $\tau$ such that 	$\tau_{n+1}=\tau\circ\theta_{\tau_n}+\tau_n,$
	$X_\tau^R\geq y_0$ and $\tau$ has a finite mean under $\Pro_{y_0}$.  Here, $\theta$ denotes the shift operator. 
\end{defn}

\begin{remark}
	\label{RemarkThresholdStrategies}
	Note that threshold strategies $R(y)$ for $y>y_0$ are admissible stationary strategies ($\tau=\tau_y:=\inf\{t\geq 0:X_t\geq y\}$). The integrability holds by general theory for diffusion processes.
\end{remark}

As a first step we show that we can restrict our attention on a compact interval of possible harvesting rates. Indeed, classical theorems on optimal stopping (and also the results in \ref{sec:imp_control_toolbox}) yield that the value of the stopping problem 
\[
\max_{\tau} \frac{\mathbb{E}[X_{\tau-}]-y_0}{\mathbb{E}[\tau]}
\]
is attained by a threshold strategy. Thus, the following result describes the maximal possible harvesting rate.

\begin{lemma}
	\label{Lemma:UniqueOptimizer}
	Let $\tilde{K} \ge 0$. Then the function 
	$k_{\tilde{K}}:(y_0, \infty) \rightarrow \mathbb{R}, \quad y \mapsto \frac{y-y_0-\tilde{K}}{\xi(y)}$, $\xi$ defined in \eqref{eq:xi},
	has a unique maximizer $\hat{y}_{\tilde{K}}$ and is strictly decreasing for all $y>\hat{y}_{\tilde{K}}$. For constants $\tilde{K}_1 < \tilde{K}_2$ we moreover have that $\hat{y}_{\tilde{K}_1} < \hat{y}_{\tilde{K}_2}$.
\end{lemma}

In order to prove this lemma we first prove the following statement which is perhaps of more general interest. See also \cite{MR2023887} for a related result. 

\begin{lemma}\label{xiconvexityLEM}
	Assume that there exists $y_1\geq y_0$ such that $\mu$ is (strictly) decreasing on $[y_1,\infty)$ and increasing on $(0, y_1]$. Then, there exists $y_2\in[y_1,\infty]$ such that $\xi$ (see equation \ref{eq:xi}) is (strictly) convex on $(y_2,\infty)$ and concave on $[y_0, y_2]$.	
\end{lemma}

\begin{proof}
	For all $x,b\in\R_+$ with $x<b$ using Dynkin's formula and standard diffusion theory, see e.g. \cite[Chapter 15.3]{KarlinTaylor}, it holds that
	\begin{align*}
		b&=\E_x[X_{\tau_b}]=x+\E_x \left[ \int_0^{\tau_b}\mu(X_t)dt \right]\\
		&=x+\int_x^b(S(b)-S(y))\mu(y)M(dy)+(S(b)-S(x))\int_0^x\mu(y)M(dy).
	\end{align*}
	Differentiating with respect to $x$ yields
	\begin{align}\label{eq:s}
		1=s(x)\int_0^x\mu(y)m(y)dy.
	\end{align}
	Now, we differentiate twice in the formula for $\xi$ in Remark \ref{rem:zeta} and obtain
	\begin{align*}
		\xi'(x)=s(x)M[0,x],\;\;\xi''(x)=s'(x)M[0,x]+s(x)m(x).
	\end{align*}
	Taking $s'(x) = \frac{2 \mu(x)}{\sigma^2(x)} s(x)$, $m(x)=2/(s(x)\sigma^2(x))$ and \eqref{eq:s} into account, we obtain
	\begin{align*}
		\xi''(x)&=-\frac{\mu(x)}{\sigma^2(x)/2}s(x)M[0,x]+\frac{1}{\sigma^2(x)/2}\\
		&=\frac{2s(x)}{\sigma^2(x)}\int_0^x(\mu(y)-\mu(x))m(y)\ dy\\
		&=:\frac{2s(x)}{\sigma^2(x)}I(x).
	\end{align*}
	Noting that $I'(x)=-\mu'(x)M[0,x]$, we see that $I$ is increasing if $\mu$ is decreasing and vise versa. Therefore, under our assumptions, $\xi''$ changes sign at most once and from negative to positive.
\end{proof}

\begin{proof}[Proof of Lemma \ref{Lemma:UniqueOptimizer}]
	Let us start by observing that
	\begin{align*}
		\lim_{y \rightarrow y_0} \frac{y-y_0- \tilde{K}}{\xi(y)} &= - \infty, \\
		\lim_{y \rightarrow \infty}  \frac{y-y_0- \tilde{K}}{\xi(y)} &= \lim_{y \rightarrow \infty} \frac{1}{\xi'(y)} = \lim_{y \rightarrow \infty} \frac{1}{2 s(y) M[0,y]} = 0
	\end{align*} and $k_{\tilde{K}}(\tilde{K} +y_0+1)>0$. Thus, a maximum point exists and is given by a critical point, i.e. by a root of 
	\begin{equation*}
		\tilde{F}_{\tilde{K}} (y) = \xi(y) - (y-y_0-\tilde{K}) \xi'(y),
	\end{equation*} since it suffices to consider the numerator of the first order condition as the denominator is always positive.
	For all $y < y_0 + \tilde{K}$ we have $\tilde{F}_{\tilde{K}}(y) >0$ since $\xi'(y)>0$ for all $y \ge y_0$. Moreover, let $y_2$ be as in Lemma \ref{xiconvexityLEM}. Then for $y$ in the (possibly empty) interval $(y_0 + \tilde{K}, y_2)$ we have
	$$\frac{\partial}{\partial y} \tilde{F}_{\tilde{K}}(y) = - \underbrace{(y-y_0-\tilde{K})}_{>0} \underbrace{\xi''(y)}_{\le 0} \ge 0.$$
	For $y > \max \{y_0 + \tilde{K}, y_2\}$ we have
	$$\frac{\partial}{\partial y} \tilde{F}_{\tilde{K}}(y) = - \underbrace{(y-y_0-\tilde{K})}_{>0} \underbrace{\xi''(y)}_{> 0} < 0.$$
	In total we obtain that $\tilde{F}_{\tilde{K}}$ is positive on $[y_0, \max \{y_0 + \tilde{K}, y_2\}]$ and is strictly decreasing after $\max \{y_0 + \tilde{K}, y_2\}$. Therefore, it has exactly one root which furthermore lies in $[\max \{y_0 + \tilde{K}, y_2\}, \infty)$, which implies that $k_{\tilde{K}}$ has exactly one maximum point and that this maximum is attained at a point in $[\max \{y_0 + \tilde{K}, y_2\}, \infty)$. Since the derivative is strictly negative for $y> \max \{y_0 + \tilde{K}, y_2\}$ and $\hat{y}_{\tilde{K}} \ge \max \{y_0 + \tilde{K}, y_2\}$ we obtain that $k_{\tilde{K}}$ is strictly decreasing for $y> \hat{y}_{\tilde{K}}$.
	
	Let now $\tilde{K}_1, \tilde{K}_2 \ge 0$ be such that $\tilde{K}_1 < \tilde{K}_2$. Furthermore, let $\hat{y}_1$ be the unique maximizer of $k_{\tilde{K}_1}$ and $\hat{y}_2$ be the unique maximizer of $k_{\tilde{K}_2}$. Then $\hat{y}_1$ solves 
	$$\xi(\hat{y}_1) - (\hat{y}_1 - y_0 - \tilde{K}_1) \xi'(\hat{y}_1)=0$$ and $\hat{y}_1 \ge \max \{y_2, \tilde{K}_1+y_0\}$ by the previous discussion. 
	Therefore, we have $\hat{y}_1 \ge y_2$, thus
	\begin{align*}
		\tilde{F}_{\tilde{K}_2}(\hat{y}_1) &= \xi(\hat{y}_1) (\hat{y}_1 - y_0 - \tilde{K}_2) \xi'(\hat{y}_1) \\
		&= \underbrace{\xi(\hat{y}_1) - (\hat{y}_1 -y_0 - \tilde{K}_1) \xi'(\hat{y}_1)}_{=0} + \underbrace{(\tilde{K}_2-\tilde{K}_1)}_{>0} \underbrace{\xi'(\hat{y}_1)}_{>0} >0.
	\end{align*}
	Since by the previous discussion $\tilde{F}_{\tilde{K}_2}$ is decreasing on $[\max \{y_2, y_0 + \tilde{K}_2\}, \infty)$ and $\tilde{F}_{\tilde{K}_2}$ is positive in the beginning of this interval, the unique root $\hat{y}_2$ of $\tilde{F}_{\tilde{K}_2}$ lies right of $\hat{y}_1$, which proves the desired claim.	
\end{proof}

Thus, we can restrict our attention to the compact interval $[0, \frac{\hat{y}_0-y_0}{\xi(\hat{y}_0)}]$ of average harvesting rates. 
Now we are in the position to conclude the description of the model by adding the following assumption:

\begin{assumption}\label{annahmenExistenzGG}
		The function $\gamma$ is of the form $\gamma(y,z)=(y-y_0)\varphi(z)$ for a continuously differentiable and strictly decreasing function  $\varphi: [0, \frac{\hat{y}_0-y_0}{\xi(\hat{y}_0)}]\rightarrow \mathbb{R}_+$.
\end{assumption}

We highlight that it would be possible to formulate more general assumptions allowing for a reward function $\gamma(y,z):[y_0, \infty) \times [0,\frac{\hat{y}_0-y_0}{\xi(\hat{y}_0)}]\rightarrow\R$ that is twice continuously differentiable in the first argument and decreasing in the second argument. However, this would yield to rather technical additional assumptions explicitly requiring the properties derived in Section \ref{sec:preliminaries}. 
Nonetheless, our restriction on this assumption has mainly economic reasons, since such a general payoff function does not have a clear economic interpretation. Indeed, in our setting of a continuum of players, a payoff function ``quantity times price'' is the only sensible formulation of the fact that agents have a negligible impact on the market price.

\section{Preliminaries}\label{sec:preliminaries}

This section collects preliminary results regarding the associated classical impulse control problems necessary to discuss the mean field game and the mean field type control problem in the subsequent sections. 

Let $f: \mathbb{R}_+ \rightarrow \mathbb{R}$ be an increasing and continuous function satisfying $f(y_0)=0$ and let $h: \mathbb{R}_+ \rightarrow \mathbb{R}_+$ be a continuous function satisfying the linear growth condition $h(x) \le c(1+x)$ for some $c>0$.
Then the control problem of interest is given as maximizing
$$\liminf_{T \rightarrow \infty} \frac{1}{T} \mathbb{E}_x \left[ \sum_{n: \tau_n \le T} (f(X_{\tau_n-}) - K) - \int _0^T h(X_s) ds \right]$$ among all admissible strategies.
To analyse the control problem define for all $x \in \mathbb{R}_+$ the auxiliary value function $$v^h(x):= \sup_{Q} \liminf_{T \rightarrow \infty} \frac{1}{T} \mathbb{E}_x \left[ \sum_{n: \tau_n \le T} (f(X_{\tau_n-}^Q) - K) -\int _0^T h(X_s^Q) ds \right]$$
for $Q=(\tau_n)_n$. We often suppress the dependence on $Q$ in the following when the context is clear. 
	
\begin{proposition}
	\label{ControlToolboxTheorem}
	Assume there is a maximizer $y^\ast$ in $[y_0,\infty)$ of
	\begin{equation}
		\label{equation:Zielfunktion}
		y\mapsto \frac{f(y)-K- \mathbb{E}_{y_0}\left[ \int_0^{\tau_y} h(X_s) ds\right]}{\xi(y)} 
	\end{equation}
for the uncontrolled process\ $X$, 
	then for all $x \in \mathbb{R}_+$ we have 
	$$v^h(x)= \sup_{y>y_0} \frac{f(y) -K- \mathbb{E}_{y_0}\left[ \int_0^{\tau_y} h(X_s) ds\right]}{\xi(y)}.$$
	Furthermore, the threshold strategy $R(y^\ast)$ is optimal for all $x \in \mathbb{R}_+$ and $R(y^\ast)$ is the unique optimizer amongst threshold strategies, if $y^\ast$ is the unique maximizer of \eqref{equation:Zielfunktion}.
\end{proposition}

The proof of this result is presented in Appendix \ref{sec:imp_control_toolbox}.

For us the central control problem of interest is for fixed $z \in [0, \frac{\hat{y}_0-y_0}{\xi(\hat{y}_0)}]$ the problem
\[
	\liminf_{T \rightarrow \infty} \mathbb{E}_x \left[ \sum_{n: \tau_n \le T} (\gamma(X_{\tau_n-},z)-K) \right]
\]
for which we now investigate the function 
$y \mapsto \frac{\gamma(y,z)-K}{\xi(y)}$:

\begin{lemma}
\label{Lemma:MaximizerProperties}
Let $\hat{y}_0$ denote the unique maximizer of $\frac{y-y_0}{\xi(y)}$. Then for every $z \in [0, \frac{\hat{y}_0-y_0}{\xi(\hat{y}_0)}]$ there exists a unique critical point $y = y_z \in (\hat{y}_0, \infty)$ of $y\mapsto\frac{\gamma(y,z)-K}{\xi(y)}$ and this is a global maximum. Moreover, for all pairs $(y,z)$ describing a critical point as given before, it holds that
$$\frac{\partial^2}{\partial y^2} \frac{\gamma(y,z)-K}{\xi(y)} <0.$$
\end{lemma}

\begin{proof}
	Fix $z \in [0, \frac{\hat{y}_0-y_0}{\xi(\hat{y}_0)}]$. Then we have
	\[
	\frac{\phi(z)(y-y_0)-K}{\xi(y)} = \phi(z) \cdot \frac{y -y_0 - \frac{K}{\phi(z)}}{\xi(y)}.
	\] Thus, the existence of a unique critical point satisfying the stated condition directly follows from Lemma \ref{Lemma:UniqueOptimizer} using $\tilde{K}= \frac{K}{\phi(z)}$. Furthermore, we have for any pair $(y,z)$ such that $y$ is a critical point for $z$ 
	\begin{align*}
		&\left( \frac{\partial}{\partial y} \right)^2 \frac{\gamma(y,z)-K}{\xi(y)} = \left( \frac{\partial}{\partial y} \right)^2 \frac{\varphi(z)(y-y_0)-K}{\xi(y)} \\
		&= \frac{\partial}{\partial y}\frac{\varphi(z) \xi(y)- \left( \varphi(z) (y-y_0) - K\right) \xi'(y)}{(\xi(y))^2} \\
		&= -\frac{ (\varphi(z)(y-y_0)-K) \xi''(y) + 2 \xi(y) \xi'(y) \overbrace{\left( \varphi(z) \xi(y)- (\varphi(z) (y-y_0)-K) \xi'(y)\right)}^{=0}}{\xi(y)^4} \\
		&= -\frac{ (\varphi(z)(y-y_0)-K) \xi''(y)}{\xi(y)^4} <0,
	\end{align*} which proves the rest of the claim.
\end{proof}

\begin{lemma}
	\label{Lemma:BRContinuous}
	The function 
	\begin{equation}
		\label{eq:BRMap}
	g: \left[0, \frac{\hat{y}_0-y_0}{\xi(y)}\right] \rightarrow [\hat{y}_0, \infty), \quad z \mapsto \argmax_{y \in [\hat{y}_0, \infty)} \frac{\gamma(y,z)-K}{\xi(y)}
	\end{equation} is continuous. In particular, the set 
	\begin{equation}
		\label{eq:SetOfCriticalY}
		\left\{y \in [y_0, \infty) \bigg| \exists z \in \left[0, \frac{\hat{y}_0-y_0}{\xi(y)}\right]: y \text{ is a critical point of } \frac{\gamma(y,z)-K}{\xi(y)}\right\}
	\end{equation} is compact.
\end{lemma}

\begin{proof}
	We apply the implicit function theorem. By Lemma \ref{Lemma:MaximizerProperties} for any fixed $z \in [0, \frac{\hat{y}_0-y_0}{\xi(\hat{y}_0)}]$ there exists a unique maximizer which moreover is a solution to 
	\begin{equation}
		\label{eq:TildeF}
		\tilde{F}(y,z) := \phi(z) \xi(y) - (\phi(z)(y-y_0)-K) \xi'(y) =0,
	\end{equation} where the right-hand side is, as stated before, the numerator of the first order condition. Since we additionally have, by Lemma \ref{Lemma:MaximizerProperties}, that 
	\[
	\frac{\partial}{\partial y} \tilde{F}(y,z) := \frac{\partial^2}{\partial y^2} \frac{\gamma(y,z)-K}{\xi(y)}<0,	
	\] we can apply the implicit function theorem and obtain that $y= \psi(z)$ for a continuous function $\psi$ defined on an open neighbourhood $U$ of $z$.
	The second part of the claim directly follows from the continuity of $g$.
\end{proof}


\section{The Mean Field Game}\label{sec:game}

In this section we prove that Assumptions \ref{ass:process} and \ref{annahmenExistenzGG} imply the existence of a  mean field equilibrium in threshold strategies and provide a criterion ensuring the uniqueness of such an equilibrium.
By definition, a mean field equilibrium in threshold strategies satisfies
\[R({y^g})\in {\argmax}_{R \text{ \footnotesize admissible }} J(R,R(y^g)).\]

In Lemma \ref{Lemma:BRContinuous} we proved that the set \eqref{eq:SetOfCriticalY}
is compact. Let now $\underline{y}$ and $\overline{y}$ denote the minimal and maximal elements of this set. Since for any equilibrium strategy of threshold type the threshold $y$ has to be optimal given the average harvesting rate $z$, we directly obtain that only a threshold from the interval $[\underline{y}, \bar{y}]$ can constitute a mean field equilibrium.

The central idea now is that any mean field equilibrium in threshold strategies is a fixed point of the following continuous function
\begin{equation}
\label{Existence:Map}
\Phi:[\underline{y},\bar{y}] \rightarrow [\underline{y},\bar{y}]; \ 	y \mapsto \argmax_{\tilde{y} \in [\underline{y},\bar{y}]} \frac{\gamma \left(\tilde{y}, \frac{y-y_0}{\xi(y)}\right) - K}{\xi(\tilde{y})}
\end{equation} and any fixed point of this map is a mean field equilibrium in threshold strategies.
Indeed, let $y$ be a threshold such that $R(y)$ is an equilibrium. By Proposition \ref{ControlToolboxTheorem}, Lemma \ref{Lemma:MaximizerProperties} and \ref{Lemma:BRContinuous} the threshold $y$ maximizes $$\tilde{y}\mapsto\frac{\gamma(\tilde{y}, \frac{y-y_0}{\xi(y)}) - K}{\xi(\tilde{y})}$$ and satisfies $\underline{y} < y<\bar{y}$.  If, on the other hand, $y \in [y_0,\bar{y}]$ is a fixed point of \eqref{Existence:Map}, then $y$ maximizes $$\tilde{y} \mapsto \frac{\gamma(\tilde{y}, \frac{y-y_0}{\xi(y)}) - K}{\xi(\tilde{y})},$$ which by Proposition \ref{ControlToolboxTheorem} and Lemma \ref{Lemma:MaximizerProperties} means that $R(y)$ is the unique optimal threshold strategy for $z= \frac{y-y_0}{\xi(y)}$, which is the defining property of a mean field equilibrium.

\begin{thm}\label{haupsatzproblem}
	Let Assumptions \ref{ass:process} and  \ref{annahmenExistenzGG} hold. Then,
	there is a fixed point of $\Phi$, which means that a mean field equilibrium in threshold strategies exists.
\end{thm}

\begin{proof}
	We note that $\Phi=g \circ f,$ where $g$ is as in \eqref{eq:BRMap} and $$f:[\underline{y},\bar{y}]\rightarrow \left[0,\frac{\hat{y}_0-y_0}{\xi(y)}\right]; \ y \mapsto \frac{y-y_0}{\xi(y)}.$$
	The function $f$ is continuous since $\xi$ is continuous and $g$ is continuous by Lemma \ref{Lemma:BRContinuous}. Therefore, the function $\Phi$ itself is continuous and Brouwer's fixed point theorem yields the existence of a fixed point.
\end{proof}

A relevant question in game theory is under which conditions an equilibrium is unique. 
As in most situations, also here the uniqueness result relies on monotonicity of the map that characterizes the equilibria. 

\begin{thm}\label{thm:unique eq}
	There is a unique mean field equilibrium in threshold strategies.
\end{thm}

\begin{proof}
	In the setting of the proof of Theorem \ref{haupsatzproblem} we have seen that the map $\Phi$ is a composition of two functions $f$ and $g$.
	The optimization problem considered in the function $g$ is to maximize $\frac{\phi(z)(y-y_0)-K}{\xi(y)}$, which is equivalent to maximizing $\frac{y-y_0-K/\phi(z)}{\xi(y)}$. Since $K/\phi(z)$ is strictly increasing in $z$ we obtain by Lemma \ref{Lemma:UniqueOptimizer} that the function $g$, that maps $z$ to the unique optimizer, is increasing. Moreover, the function $f$ is decreasing since by Lemma \ref{Lemma:MaximizerProperties} we have $\underline{y}>\hat{y}_0$ and by Lemma \ref{Lemma:UniqueOptimizer} the function $\frac{y-y_0}{\xi(y)}$ is strictly decreasing for $y \ge \hat{y}_0$. Thus, in total the function $\Phi$ is strictly decreasing, which yields in combination with Theorem \ref{haupsatzproblem} that a unique fixed point exists.
\end{proof}

\section{The Mean Field Control Problem}\label{sec:problem}

We now consider the situation that the market participants cooperate in the sense that they agree to choose a common strategy. 
The mean field type control problem therefore consists of maximizing
\begin{equation}\label{eq:mean_field_problem}
J_x(R,R)= \liminf_{T\rightarrow\infty} \frac{1}{T}\E_x\left[\sum_{n:\tau_n\leq T} \left(\gamma \left(X_{\tau_n-}^R,\frac{\mathbb{E}[\hat{X}_\tau^R]-y_0}{\mathbb{E}[\tau]}\right)-K\right)\right]\end{equation}
over all admissible stationary strategies $R$ with corresponding stopping times $\tau$.
In the class of threshold strategies $R(y)$,  $y\geq y_0$, the optimization problem is easily solved. In this case, we explicitly know
$\frac{\mathbb{E}[\hat{X}_\tau^R]-y_0}{\mathbb{E}[\tau]} = \frac{y-y_0}{\xi(y)},$ hence, we just have to maximize the explicitly given real function
\[v(y):=J_x(R(y),R(y))= \frac{\gamma \left(y,\frac{y-y_0}{\xi(y)} \right)-K}{\xi(y)}.\]
However, the problem \eqref{eq:mean_field_problem} is a non-standard stochastic control problem due to the expectation-terms. Therefore, it is by no means clear that threshold strategies are indeed optimal in the class of all admissible stationary strategies. In the following theorem, we prove this fact by splitting up the problem as follows
\begin{equation}
	\label{Problem:AlternativeForm}
	\sup_R J_x(R,R)=\sup_{z\in[0,\frac{\hat{y}_0-y_0}{\xi(y)}]}\sup_{R\mbox{ \footnotesize with } \frac{\mathbb{E}[X_\tau^R]-y_0}{\mathbb{E}[\tau]}=z}\liminf_{T\rightarrow\infty} \frac{1}{T}\E_x\left[\sum_{n:\tau_n\leq T} (\gamma(X_{\tau_n-}^R,z) - K)\right]
\end{equation}
and then utilizing a Lagrange-type approach to reduce the restricted problem to a standard problem. 

%
%
%


\begin{thm}
\label{Theorem:ControlProblem}
Under Assumptions \ref{ass:process} and \ref{annahmenExistenzGG}
the value of the optimization problem \eqref{eq:mean_field_problem} is 
\[\sup_yH(y),\;\;H(y)=\frac{\gamma \left(y,\frac{y-y_0}{\xi(y)}\right)-K}{\xi(y)}\]
and if $y^*$ is a maximizer of $H$, then the threshold strategy $R=R(y^\ast)$ is optimal in the class of all stationary strategies. Moreover, any such threshold value $y^\ast$ lies in $[\hat{y}_0,\infty)$.
\end{thm}

\begin{proof}
	By the previous discussion, it is enough to consider, for fixed $z\in\left[0, \frac{\hat{y}_0-y_0}{\xi(\hat{y}_0)}\right]$, the restricted problem
	\begin{equation}\label{eq:restr_problem}
	\sup_{R\mbox{ with } \frac{\mathbb{E}[X_\tau^R]-y_0}{\xi(y)} =z}\liminf_{T\rightarrow\infty} \frac{1}{T}\E_x\left[\sum_{n:\tau_n\leq T}(\gamma(X_{\tau_n-}^R,z)-K)\right]
	\end{equation}
	and prove that a threshold strategy is optimal. 
	
	Following a standard Lagrange approach, we consider, for fixed $\lambda < \phi(z)$, the associated unconstrained  problem
		\begin{equation}\label{eq:lagrange}
	\sup_{R}\liminf_{T\rightarrow\infty} \frac{1}{T}\E_x\left[\sum_{n:\tau_n\leq T}(\gamma(X_{\tau_n-}^R,z)-K)\right]-\lambda\left(\frac{\mathbb{E}[\hat{X}_{\tau-}^R]-y_0}{\mathbb{E}[\tau]} -z\right).
	\end{equation}
	Writing $N(t) = \max \{n: \tau_n \le t\}$ we obtain by standard renewal results that	
	\begin{align*} \frac{1}{T} \sum_{n : \tau_n \le T} \left( X_{\tau_n-}^R- y_0 \right)
	&= \frac{1}{T} \sum_{n=1}^{N(T)} \left( X_{\tau_n-}^R- y_0 \right) = \frac{N(T)}{T} \cdot \frac{1}{N(T)} \sum_{n=1}^{N(T)} \left( X_{\tau_n-}^R- y_0 \right) \\
	&\rightarrow\frac{1}{\mathbb{E}[\tau]} \cdot \left( \mathbb{E}[X_{\tau-}] -y_0\right) = \frac{1}{\mathbb{E}[\tau]} \cdot \left( \mathbb{E}[\hat{X}_{\tau-}] -y_0\right) 
	 \end{align*}
	a.s. and in $L^1$, hence problem \eqref{eq:lagrange} may be rewritten as
		\begin{equation*}
\sup_{R}\liminf_{T\rightarrow\infty} \frac{1}{T}\E_x\left[\sum_{n:\tau_n\leq T}\left( (\phi(z)-\lambda)  (X_{\tau_n-}-y_0) - K\right)\right] + \lambda z.
\end{equation*}	
As $\lambda z$ is just a constant and 
\[
	\frac{(\phi(z)-\lambda)(y-y_0)-K}{\xi(y)} = (\phi(z) - \lambda)\frac{y-y_0- \tilde{K}}{\xi(y)}
\] with $\tilde{K} = K/(\phi(z)- \lambda)$ we obtain by Lemma \ref{Lemma:UniqueOptimizer}, and Proposition \ref{ControlToolboxTheorem} that there is a  threshold $y=y(\lambda, z) \ge \hat{y}_0$ such that $R(y)$ is the unique optimal threshold strategy. Moreover, by Lemma \ref{Lemma:UniqueOptimizer} the optimizer is increasing in $\lambda$ and satisfies $y(\lambda, z) \ge y_0 + \frac{K}{\phi(z)-\lambda}$. Thus, we obtain 
\[
\lim_{\lambda \rightarrow \phi(z)} y(\lambda,z) = \infty \quad \text{and} \quad \lim_{y \rightarrow -\infty} y(\lambda,z) = \hat{y}_0.
\] 
Moreover, we can show as in Lemma \ref{Lemma:BRContinuous} that the map $\lambda \mapsto y(\lambda,z)$ is continuous. Therefore, for any $z \in (0,\hat{y}_0)$, there is a $\lambda_z$ such that 
\[
	z = \frac{y(\lambda_z,z)-y_0}{\xi(y(\lambda_z,z))}.
\]
Since $R(y(\lambda_z,z))$ is an (unconstrained) maximizer for 
\[
\sup_{R} \liminf_{T\rightarrow\infty} \frac{1}{T} \mathbb{E}_x \left[ \sum_{n: \tau_n \le T} \gamma(X_{\tau_n-},z) -K \right] - \lambda_z \left( \frac{\mathbb{E}\left[X_{\tau-}\right] -y_0}{\mathbb{E}[\tau]}-z \right)
\]
and fulfils $ z= \frac{y(\lambda_z,z) -y_0}{\xi(y(\lambda_z,z))} $ it is a maximizer for \eqref{eq:restr_problem} as well.
We conclude by noting that any threshold startegy $R(y)$ with threshold $y<\hat{y}_0$ is not optimal for \eqref{eq:restr_problem} since $\phi(z)(y-y_0)< \phi(z)(y(\lambda_z,z)-y_0)$.
\end{proof}

\section{Computation and Comparison of Solutions}\label{sec:comparison}

We have shown that, both for the mean field game and for the mean field equilibrium, there is a solution being a threshold strategy. In both cases the threshold can be obtained by maximizing the function $$G:(y_0,\infty)\times \R_+;\ \ (x_1,x_2)\mapsto \frac{\gamma \left(x_1,\frac{x_2-y_0}{\xi(x_2)}\right)-K}{\xi(x_1)}$$ in a certain way. 
We remark that $G$ is differentiable, since $x_2\mapsto \frac{x_2-y_0}{\xi(x_2)}$ is differentiable due to \eqref{DensityOfStationaryProcess}.

The optimal thresholds $y^p$ for the cooperative mean field type control problem are given as the solutions of $$0=\frac{\partial}{\partial y}G(y,y)=\frac{\partial}{\partial x_1}G(y,y)+\frac{\partial}{\partial x_2}G(y,y)$$ and the threshold $y^g$ constituting an equilibrium for the competitive mean field game is given as  the solution of $$0=\frac{\partial}{\partial x_1}G(y,y).$$ 

Our assumptions, in particular the assumption that $\gamma(x,z)$ is strictly decreasing in $z$ for all $x$,  yield the following comparison result stating that the threshold under competition is smaller than in the cooperative regime. Therefore, in our (oversimplified) model, competition leads to a smaller average volume of wood in the forest stands.
\begin{thm}
	Let $y^p$ be an optimal threshold for the mean field type control problem and $y^g$ the unique threshold describing a mean field equilibrium. Then $y^p \ge y^g$.
\end{thm}

\begin{proof}
	Assume that there is a threshold for the mean field type control problem $y^p$ such that $y^g > y^p$. 
	Since by Lemma \ref{Lemma:UniqueOptimizer} and Theorem \ref{Theorem:ControlProblem}  we have $y^g,y^p > \hat{y}_0$ Lemma \ref{Lemma:UniqueOptimizer} yields that $\frac{y^g -y_0}{\xi(y^g)} < \frac{y^p-y_0}{\xi(y^p)}$.
	As $\gamma(x,z)$ is strictly decreasing in $z$, it holds that  $G(x,y^g)>G(x,y^p)$ for all $x >y_0$.		
	Since $y^g$ is an equilibrium we obtain that $$G(y^p,y^g) \le G(y^g,y^g).$$ Since $y^p$ is a solution of the mean field type control problem we have $$G(y^g,y^g) \le G(y^p,y^p).$$ All in all we obtain $$G(y^p,y^p)< G(y^p,y^g) \le G(y^g,y^g) \le G(y^p,y^p),$$ which is a contradiction.
\end{proof}

\section{An alternative market state dependence} \label{sec:EW}

In this section we investigate a different interaction mechanism. More precisely, we assume that the prices do not depend on the average harvesting rate, but on the expected wood supply. The existence of mean field equilibria and optimal mean field type control can be established using similar methods. However, we no longer obtain that equilibria in threshold strategies are unique and the comparison result now holds with interchanged roles.
We do not completely derive the results here, but instead we highlight, at which points differences in the results or proofs occur.

In this section we assume that the wood price at time $t$ depends on the volume of wood available at that time, which is described by the mean $\mathbb{E}[\hat{X}_t^Q]$ depending on the initial distribution of $\hat{X}_0^Q$. 
As we consider a long-term reward structure, it is reasonable to assume that the process $\hat{X}^Q$ has an invariant distribution $\Pi^Q$ and $\hat{X}^Q$ is started from this. We call strategies satisfying this assumption \textit{invariant admissible strategies}.
In this setting, $\mathbb{E}[\hat{X}_t^Q] = \int x \Pi^Q(dx)=:\mathbb{E}[\hat{X}_\infty^Q]$ is independent of $t$. 
Therefore, the new reward functional for each pair of admissible strategies $R, \ \Qs$, with $\Qs$ invariant, is given by 
\begin{align}\label{eq:J_x2}
J_x(R,\Qs):= \liminf_{T\rightarrow\infty} \frac{1}{T}\E_x\left(\sum_{n:\tau_n\leq T}(\gamma(X_{\tau_n-}^R,\mathbb{E}[\hat{X}^\Qs_\infty])-K)\right),
\end{align}
where $\gamma(x,z)$ is now the payoff function that models the reward the decision maker gets each time harvesting, which we assume here to depend on the average amount of wood $\mathbb{E}[\hat{X}_\infty^\Qs]$ that the other market participant have in their forest.

Except for the change in the reward functional the formulation of the mean field game and the mean field type control problem are completely analogous and also the standing assumptions \ref{ass:process} are unchanged. Only for the payoff assumption \ref{annahmenExistenzGG} we have to adjust the range of possible values for the expected size of the forest stand. This is done in the following lemma:

\begin{lemma}\label{lem:bounds_z}
	Let\ $R$ be an admissible impulse control strategy. Then
	\[z_1\leq \liminf_{T\rightarrow\infty}\E [X_T^R] \leq \limsup_{T\rightarrow\infty}\E [X_T^R] \leq z_2,\]
	where 
	$z_1:=\E [X_\infty^{r}]$ for the diffusion process $X^r$ reflected downwards in\ $y_0$ and $z_2:=\E [X_\infty]$ for the uncontrolled diffusion process $X$.
\end{lemma}

\begin{proof}
	First note that the expectations $z_1,z_2$ exist by Assumption \ref{ass:process} and as $X^r$ is positively recurrent.
	The inequalities can be proved using an easy (partial) coupling argument: \\
	We construct a version of $X^{r}$ by letting it run coupled with $X^R$ until a state $\geq y_0$ is reached. 
	Then, we reflect $X^{r}$ in\ $y_0$ downwards and let both processes run following their dynamics with respect to the same Brownian motion until the first time the two paths meet again. 
	Then, we couple the paths and follow this rule. 
	Consequently, for each $t$ and each\ $\omega$, we have $X^{r}_t(\omega)\leq X^R_t(\omega)$, proving the first inequality. \\
	Similarly, we construct a version of the uncontrolled diffusion $X$ by running coupled with $X^R$ until the first impulse time. Then we let both processes run following their dynamics with respect to the same Brownian motion until we couple them the next time the two paths meet and so on. Again, for each $t$ and each\ $\omega$, we have $X^R_t(\omega)\leq X_t(\omega).$
\end{proof}

With this preparation we can restate Assumption \ref{annahmenExistenzGG}:

\begin{assumption}\label{annahmenExistenzGGEW}
		The function $\gamma$ is of the form $\gamma(y,z)=(y-y_0)\varphi(z)$ for a continuously differentiable and strictly decreasing function  $\varphi:[z_1,z_2] \rightarrow \mathbb{R}_+$.
\end{assumption}

As a next step we investigate the interaction term $\mathbb{E}[X_\infty^Q]$ for an economically reasonable class of strategies that moreover lead to a stationary controlled process. This class of strategies is the set of all admissible stationary strategies with the additional requirement that $\tau$ is non-lattice. We remark, that threshold strategies $R(y)$ with $y>y_0$ satisfy this requirement.
	
	\begin{proposition}
		For each admissible stationary strategy $R$ with corresponding stopping time $\tau$ such that $\tau$ is non-lattice,  the stationary distribution for the process $X^R$, denoted by $\Pi$, exists and is given by
		\[ \int f(x)\Pi(dx)=\frac{1}{\E_{y_0}[\tau]}\E_{y_0} \left[\int_0^\tau f(X_s)ds \right]. \]
	\end{proposition}
	\begin{proof}
		It is immediately seen that $X^R$ is a regenerative processes in the sense of \cite{asmussen2008applied}, Chapter VI. Therefore, the result holds by ibid, Theorem 1.2 on p.170.
	\end{proof}

	In the case that $R=R(y)$ is a threshold strategy, a more explicit description of the limiting distribution is possible.
	By standard diffusion theory, see e.g. \cite{helmes2017continuous}, Proposition 3.1, we have that $X^{R(y)}$ has a stationary distribution with density 
	\begin{align}
		\label{DensityOfStationaryProcess}
		\pi_{y_0,y} (x) = \begin{cases}
			0, &x>y \\ \kappa m(x) S[x,y], &x \in [y_0,y] \\  \kappa m(x) S[y_0,y] &x \le y_0 
		\end{cases},
	\end{align} where $S[x,y]:=S(y)-S(x)$ denotes the Stieltjes measure and $$\kappa = \left(\int_ {y_0}^y S[w,y]\ M(dw) + S[y_0,y]M[0,y_0) \right)^{-1}.$$
	Using this we obtain
	\[
	\mathbb{E} \left[ X_\infty^{R(y)} \right] = \int_{-\infty}^y x \pi_{y_0,y}(x) \text{d} x,
	\] which in particular yields that the map $y \mapsto \mathbb{E} \left[ X_\infty^{R(y)} \right]$ is continuous.
	
	Moreover, we also obtain in this setting a monotonicity result. However, whereas the harvesting rate was decreasing in the threshold value (at least for those thresholds of interest), here we obtain that the expected volume of wood is increasing in the threshold. 
		
	\begin{lemma}
		\label{ExpectationIncreasing}
		For all $y_1< y_2$ it holds that $X_\infty^{R(y_1)}<_{st} X_\infty^{R(y_2)}$, where $<_{st}$ denotes the stochastic  ordering.\\
		In particular, $\mathbb{E}[X_\infty^{R(y)}]$ is increasing in the threshold level $y$.
	\end{lemma}
	
	\begin{proof}
		We show that given an arbitrary pair $y_1 < y_2$ there is a switching point $z\in[y_0,y_1]$ such that for the corresponding densities it holds that
		$$\pi_{y_0, y_1} (x) > \pi_{y_0, y_2} (x) , x < z, \quad \pi_{y_0, y_1}(x) \le \pi_{y_0, y_2}(x), x \ge z.$$ This immediately yields the statement.
		To this end, we first prove that for fixed $y_0$ and $x \le y_0$ the density $\pi_{y_0,y}(x)$ is decreasing in $y$:
		We have $$\pi_{y_0,y}(x) = m(x) \frac{g_1(y)}{f_1(y)+f_2(y)}$$ 
		with 
		$$f_1(y) = \int_{y_0}^y S[w,y] dM(w), \quad f_2(y) = S[y_0,y] M[0,y_0],\quad g_1(y) = S[y_0,y]$$
		Using $$f_1'(y) = s(y) M[y_0,y], \quad f_2'(y) = s(y) M[0,y_0], \quad g_1'(y) = s(y)$$ 
		we obtain
		\begin{align*}
			\frac{\partial}{\partial y} \pi_{y_0,y}(x) 
			&= m(x) \frac{g_1'(y) (f_1(y)+f_2(y))-g_1(y) (f_1'(y)+f_2'(y))}{(f_1(y)+f_2(y))^2} \\
			&= \frac{m(x)}{(f_1(y)+f_2(y))^2} s(y) \Bigg[ \int_{y_0}^y S[w,y] dM(w) + S[{y_0},y] M[0,{y_0}] \\
			&\quad   - S[{y_0},y] M[{y_0},y] - S[{y_0},y] M[0,{y_0}] \Bigg] \\
			&= \frac{m(x)}{(f_1(y)+f_2(y))^2} s(y) \left[ \int_{y_0}^y \underbrace{(S[w,y]-S[{y_0},y])}_{=-S[{y_0},w]} dM(w) \right] <0.
		\end{align*} 
		This yields $\pi_{{y_0},y_1}(x)>\pi_{{y_0},y_2}(x)$ for all $x\leq y_0$. It remains to consider the case  $x> y_0$. We first show that for $y\in[y_0,y_1]$ and $x \in [{y_0},y]$ the derivative $\frac{\partial}{\partial y} \pi_{{y_0},y}(x)$ may be decomposed as follows:
		$$\frac{\partial}{\partial y} \pi_{{y_0},y}(x)=m(x)h(x,y),$$ where $h(x,y)$ is increasing in $x$. Indeed, for all $x \in [{y_0},y]$ using the notation $g_2(y) = S[x,y]$ we get
		\begin{align*}
			\frac{\partial}{\partial y} \pi_{{y_0},y}(x) &= m(x) \frac{g_2'(y) (f_1(y)+f_2(y)) - g_2(y)(f_1'(y)+f_2'(y))}{(f_1(y)+f_2(y))^2} \\
			&= \frac{m(x) s(y)}{(f_1(y)+f_2(y))^2} \Bigg[ \int_{y_0}^y S[w,y] dM(w) + S[{y_0},y] M[0,{y_0}]  \\
			&\quad  -S[x,y] M[{y_0},y]- S[x,y] M[0,{y_0}] \Bigg]\\&= m(x)\left\{
			\frac{1}{(f_1(y)+f_2(y))^2}\left(\int_{y_0}^y S[w,x]dM(w)+S[y_0,x]M[0,y_0]\right)\right\}\\
			&=:m(x)h(x,y),
		\end{align*}	
		where $h(x,y)$ is indeed obviously increasing in $x$.
		This decomposition is sufficient as it yields that $$\pi_{{y_0},y_2}(x) - \pi_{{y_0},y_1}(x) = m(x)\int_{y_1}^{y_2} h(x,y) \text{d} y$$ changes sign just once. Hence, using $\pi_{{y_0},y_1}(x)>\pi_{{y_0},y_2}(x)$ for $x\le y_0$ and $\pi_{{y_0},y_1}(x)=0<\pi_{{y_0},y_2}(x)$ for $x\in(y_1,y_2)$, there exists some $z\in[y_0,y_1]$ satisfying above conditions.\\	
	\end{proof}

	\subsection{The Mean Field Game}
	
With these preparations we can analyse the existence of mean field equilibria in threshold strategies:

\begin{thm}
	Let Assumptions \ref{ass:process} and \ref{annahmenExistenzGGEW} hold. Then a mean field equilibrium exists.
\end{thm}

\begin{proof}
	As in the case of average harvesting rates a mean field equilibrium in threshold strategies is given by a fixed point of the slightly modified map 
	\[
	\Phi: [y_0, \bar{y}] \rightarrow [y_0,\bar{y}], \quad y \mapsto \argmax_{\tilde{y} \in [y_0,\bar{y}]} \frac{\gamma \left(\tilde{y}, \mathbb{E}\left[ X_\infty^{R(y)}\right]\right)}{\xi(\tilde{y})}.
	\]
	This function can be again represented as a composition of two functions $f$ and $g$, where 
	\[
	f:[y_0, \bar{y}] \rightarrow [z_1,z_2], \quad y \mapsto \mathbb{E}_x\left[ X_\infty^{R(y)} \right]
	\]
	and $g$ is defined as in \eqref{eq:BRMap} with the domain now given by $[z_1,z_2]$. By previous discussions the map $f$ is continuous and also the continuity of $g$ can be proved as in Lemma \ref{Lemma:BRContinuous}. Thus, a fixed point of the map $\Phi$ and thus a mean field equilibrium exists.
\end{proof}

In contrast to the case of average harvesting rates here we do not obtain a general uniqueness criterion under assumption \ref{annahmenExistenzGGEW}. Indeed, we obtain that the function $\Phi$ is increasing instead of decreasing when we assume that $\phi$ is additionally differentiable: Recall the notation from Lemma \ref{Lemma:BRContinuous}. By the implicit function theorem we have that 
\[\frac{\partial g(z)}{\partial z} = - \frac{\frac{\partial}{\partial x_2} \tilde{F}(g(z),z)}{\frac{\partial}{\partial x_1} \tilde{F}(g(z),z)}. \]
By Lemma \ref{Lemma:MaximizerProperties} we have $\frac{\partial}{\partial x_1} \tilde{F}(g(z),z)<0$. Moreover, by definition any critical point $g(z)$ given $z$ satisfies
\[
	\xi(g(z)) - (g(z)-y_0)\xi'(g(z)) = -\frac{K}{\phi(z)} \xi'(g(z)),
\] which yields
\begin{align*}
	\frac{\partial}{\partial x_2} \tilde{F}(g(z),z)
	&= \phi'(z) \left( \xi(g(z))-(g(z)-y_0)\xi'(g(z)) \right) \\
	&= - \phi'(z) \frac{K}{\phi(z)} \xi'(g(z))>0.
\end{align*} Thus, $g$ is increasing. Since $f$ is also increasing, we obtain that $\Phi$ itself is increasing. 

\subsection{The Mean Field Type Control Problem}

The analysis for the mean field type control problem can be carried out similarly as before. The central technical difference is that now an auxiliary problem with running costs has to be investigated. To this end, we need the following assumption:

\begin{assumption}\label{annahmenExistenzGG3}
	\begin{enumerate}
		\item For every $z\in [z_1,z_2]$ and every $\lambda\in(0,\infty)$ there is a unique critical point $y=y_z\in(y_0,\infty)$ of $$\frac{\gamma(y,z)-K- \lambda \mathbb{E}_{y_0}\left[\int_0^{\tau_y} X_s \text{d} s \right]}{\xi(y)}$$ and this is a global maximum. 
		\item  For all pairs  $(y,z)$ describing a critical point as given before, it holds 
		$$\frac{\partial^2}{\partial y^2}\frac{\gamma(y,z)-K-\lambda \mathbb{E}_{y_0}\left[ \int_0^{\tau_y}X_s \text{d} s\right]}{\xi(y)}<0$$
	\end{enumerate}
\end{assumption}

\begin{thm}
	Under Assumptions \ref{ass:process}, \ref{annahmenExistenzGGEW}, and \ref{annahmenExistenzGG3}, the value for the problem of
	maximizing
	\begin{equation}\label{eq:mean_field_problem_old}
		J_x(R,R)= \liminf_{T\rightarrow\infty} \frac{1}{T}\E_x\left[\sum_{n:\tau_n\leq T} \left(\gamma \left(X_{\tau_n-}^R,\frac{\mathbb{E}[\hat{X}_\tau^R]-y_0}{\mathbb{E}[\tau]}\right)-K\right)\right]\end{equation}
 is 
\begin{align*}\label{Problem:old}
		\sup_yH(y),\;\;H(y)=\frac{\gamma \left(y,\mathbb{E} \left[ X_\infty^{R(y)}\right] \right)-K}{\xi(y)}
\end{align*}
	and if $y^*$ is a maximizer of $H$, then the threshold strategy $R=R(y^\ast)$ is optimal in the class of all stationary strategies. 
\end{thm}

\begin{proof}
	As in the case of average harvesting rates it suffices to consider for fixed $z \in [z_1, z_2]$ the restricted problem 
	\begin{equation}\label{eq:restr_problemEW}
		\sup_{R\mbox{ with } \mathbb{E}[\hat{X}_\infty^R]=z}\liminf_{T\rightarrow\infty} \frac{1}{T}\E_x\left[\sum_{n:\tau_n\leq T}(\gamma(X_{\tau_n-}^R,z)-K)\right]
	\end{equation}
	and prove that a threshold strategy is optimal. Also here we follow the Lagrange approach and consider now for a fixed $\lambda \ge 0$ the associated unconstrained problem
	\begin{equation}\label{eq:lagrangeEW}
		\sup_{R}\liminf_{T\rightarrow\infty} \frac{1}{T}\E_x\left[\sum_{n:\tau_n\leq T}(\gamma(X_{\tau_n-}^R,z)-K)\right]-\lambda\left(\mathbb{E}[\hat{X}_\infty^R]-z\right).
	\end{equation}
	Using standard caculus we obtain 
	\[\mathbb{E}\left[\hat{X}_\infty^R\right]=\lim_{T\rightarrow\infty}\E \left[\hat X^{ R}_T \right] =\lim_{T\rightarrow\infty}\E_x  \left[X^{ R}_T \right] =\lim_{T\rightarrow\infty}\frac{1}{T}\E_x \left[\int_0^T X^R_tdt\right] , \]
	and hence, we can rewrite\eqref{eq:lagrangeEW} as
	\begin{equation*}
		\sup_{R}\liminf_{T\rightarrow\infty} \frac{1}{T}\E_x\left[\sum_{n:\tau_n\leq T}(\gamma(X_{\tau_n-}^R,z)-K)-\int_0^T \lambda X^R_t dt\right] + \lambda z.
	\end{equation*}	
	
	Using Proposition \ref{ControlToolboxTheorem} and Assumption \ref{annahmenExistenzGG3} we obtain that there is a unique threshold $y=y(\lambda,z)$.
	Let $y(z)$ be a threshold value satisfying $z= \mathbb{E}[X_\infty^{R(y)}]$. Then we obtain that all maximizers $z$ of \eqref{eq:restr_problemEW}  fulfil $y(z) \le y(0,z)$.
	Indeed, assume that $y(z) > y(0,z)$. Since $\gamma(y,z)$ is decreasing in the second component, we obtain for the value function $J_x$ from \eqref{eq:J_x2}
	$$J_x(R(y(z)),R(y(z))) \le J_x(R(y(0,z)), R(y(z)) \le J_x(R(y(0,z)),R(y(0,z))),$$
	which is a contradiction.
	Moreover, 
	\[\lim_{\lambda\to\infty}y({\lambda,z})=y_0\leq  y(z) \]
	and, again by the implicit function theorem and Assumption \ref{annahmenExistenzGG3}, the function $\lambda \mapsto y(\lambda,z)$ is continuous.
	Due to Lemma  \ref{ExpectationIncreasing}
	\[\lim_{\lambda\to\infty}\int w\Pi^{R(y(\lambda,z))}(dw)\leq z\leq \lim_{\lambda\to 0}\int w\Pi^{R(y(\lambda,z))}(dw).\]
	Therefore, there is a $\lambda_z$ such that 
	$z=\int w\pi^{R(y(\lambda_z,z))}(dw)$. Now the proof can be concluded as in the average harvesting rate case.
\end{proof}

\subsection{Computation and Comparison}

As in the case of the average harvesting rate we obtain the thresholds by maximizing now the function 
$$G:(y_0,\infty)\times \R_+;\ \ (x_1,x_2)\mapsto \frac{\gamma(x_1,\mathbb{E}[\hat{X}_\infty^{R(x_2)}])-K}{\xi(x_1)}$$
in those ways described in Section \ref{sec:comparison}. We remark that also here $G$ is differentiable, since $x_2\mapsto\mathbb{E}\left[ \hat{X}_\infty^{R(x_2)} \right]$ is differentiable due to \eqref{DensityOfStationaryProcess}.

Interestingly, also in this setting we obtain a comparison result, however, the relations hold in the different direction: 

\begin{thm}
	Let $y^p$ be an optimal threshold for the mean field type control problem and $y^g$ be a threshold describing a mean field equilibrium. Then $y^p \le y^g$.
\end{thm}
	
\begin{proof}
	Since $\gamma(x,z)$ is strictly decreasing in $z$ and $x \mapsto \mathbb{E}[X_\infty^{R(x)}]$ is increasing (see Lemma \ref{ExpectationIncreasing}), we obtain that $G(x,z)$ is decreasing in $z$. Assume that there is an equilibrium threshold $y^g$ and an optimal threshold for the mean field type control problem $y^p$ such that $y^g < y^p$. Then we obtain $$G(y^g,y^g)\le G(y^p,y^p) < G(y^p,y^g) \le G(y^g,y^g),$$ which is a contradiction.
\end{proof}
	
\section{Example}\label{sec:examples}
To illustrate our results, first, we consider the case of a classical logistic stochastic growth model. This is, the controlled process follows the dynamics
\[dX_t=X_t(a-bX_t)dt+\beta X_tdW_t,\]
where $a,b,\beta$ are positive constants. This diffusion is well-studied. We refer to \cite{MR3502394} for the results we use here and further references. The results there also yield that our assumptions are fulfilled. 

Using the notation $q:=1/2-a\beta^{-2}$, it is well-known that $X$ converges towards a unique stationary distribution if $q< 0$ and converges to 0 a.s. for $q>0$. We assume $q<0$ in the following.
Speed measure and scale function are, resp., given by the densities
\begin{align}
s(x)&=x^{2q-1}\exp\left(\frac{2}{\beta^2}b(x-1)\right)\label{eq:s_ex1}\\
m(x)&=\frac{2}{\beta^2}x^{-2q-1}\exp\left(-\frac{2}{\beta^2}b(x-1)\right).\label{eq:s_ex2}
\end{align}
In this case, the function $\xi$ is known (semi-) explicitly:
\[\xi(y)=\frac{1}{\beta^2|{q}|}\left(\log\left(\frac{y}{y_0}\right)+\sum_{n=1}^\infty\frac{1}{(1-2q)_n}\frac{(\rho y)^n}{n}-\sum_{n=1}^\infty\frac{1}{(1-2q)_n}\frac{(\rho y_0)^n}{n}\right),\]
where $(u)_n=u(u+1)\cdots(u+n-1)$ denotes the Pochhammer symbol and $\rho:=2b\beta^{-2}$. (The series may be represented using hypergeometric functions.) 
For each\ $y$ the expectation
	\[\mathbb{E}[X_\infty^{R(y)}]=\int_{-\infty}^y x \pi_{y_0,y} (x) \text{d}x\] 
	needed for the model considered in Section \ref{sec:EW} 
	can be calculated according to \eqref{DensityOfStationaryProcess} using \eqref{eq:s_ex1} and \eqref{eq:s_ex2}.

To study both models introduced above in parallel, we write
\[c(y):=\begin{cases}
	\frac{y-y_0}{\xi(y)}&\mbox{for the main model},\\
	\E \left[X^{R(y)}_\infty \right] &\mbox{for the model in Section \ref{sec:EW}.}
\end{cases}\]
For the mean field games all we have to find are thresholds $y^g$ such that
\[y^g=\argmax_y \frac{\gamma \left(y,c(y^g)\right)-K}{\xi(y)}.\]
For the mean field type control problem, the function
\begin{equation}\label{eq:ex_probl}
y\mapsto \argmax_y \frac{\gamma \left(y,c(y)\right)-K}{\xi(y)}
\end{equation}
has to be optimized.
As all expressions are known explicitly, this task can be carried out straightforwardly. Here we use the following set of parameters 
$$q := -1,\,b := 1/2,\,\beta := 1,\,y_0 := 1,\,\, K := 1.
$$

\subsection{A reward function yielding unique solutions}

Let us first choose $\gamma(y,z)=(y-y_0)/(z+1)$ for the main model. Then we obtain for the game a unique equilibrium threshold $y^g\approx 5.13$ with corresponding value $0.243$, cf. Theorem \ref{thm:unique eq}.  For the mean field type control problem the optimizer is $y^p\approx 5.9$ with value $0.254$. As discussed in Section \ref{sec:comparison}, the threshold $y^p$ is higher than\ $y^g$. 
The value in the mean field type control problem is of course higher than the value in the game.

\subsection{A reward function yielding multiple equilibria}\label{subsec:multiple}
As mentioned above, the model discussed in Section \ref{sec:EW} potentially allows for multiple equilibria. Indeed, let us now consider the reward function
$$\gamma:[y_0,\infty)\times \R_+\to\R_+; \ (y,z)\mapsto \frac{y-y_0}{1+\exp(10(z-1.9) )}.$$ Having this logistic dependence on $z$ yields three equilibria which are approximately at the points $y^g_1\approx 4.55$, $y^g_2\approx 6.8$ and $y^g_3\approx 55.5$. While the first and the last one of them are stable in the sense that when starting with a value $y_1$ in an interval around the equilibrium point the iteration used to numerically determine the equilibrium points defined by\[ y_{n+1}=\argmax_{\tilde y} \frac{\gamma \left(\tilde y,\E X^{R(y_n)}_\infty \right)-K}{\xi(\tilde{y})}\] for all $n \in \N$ will converge to $y^g_1$ and $y^g_3$ respectively, this is not the case for $y^g_2$.

\appendix
\section{Solving the Auxiliary Control Problem(s)}\label{sec:imp_control_toolbox}

In this section we  prove Proposition \ref{ControlToolboxTheorem}. First, we derive a verification result. Thereafter, we present a candidate for the value function and the optimal threshold by relying on an associated stopping problem, for which we then prove that it indeed satisfies the conditions of the verification result.

\begin{lemma}[Verification result]\label{veridrei}
	\begin{enumerate}[(i)]
		\item
	 Let $g$ be a measurable function on $\R_+$, let $u$ be defined by \[u(x,y)=f(x)-f(y)-K-g(x)+g(y)\] for all $x, y\in \R_+$, $y\leq x$, and assume  \begin{enumerate}
			
			\item  $M=({g(X_t)}-\int_0^t (h(X_s)+ \rho)   ds)_{t\geq0}$ is a supermartingale  under $\Pro_{x}$ for all $x\in \R_+$,
			\item \[\limsup_{T\rightarrow\infty}\frac{\E_{x}g(X^\Qs_T)}{T}\geq 0
			\text{ for all admissible } \Qs, x\in \R_+, \]
			\item \[u(x,y_0) \leq 0 \text{  for all } x\in \R_+, \]
		\end{enumerate} Then
		\[v^h(x)\leq \rho \text{ for all } x\in \R_+.\]
		
		\item\label{item:b} If furthermore $\Qs^*=(\tau_n^*)_{n \in \N}$ is admissible and such that
		\begin{enumerate}
			\item Using the notation $M^\Qs_t:={g(X^\Qs_t)}-\int_0^t (h(X^\Qs_s)+ \rho) \   ds$ for all $Q$, we have 
			 $$\E_{x}\left(M_{\tau^*_{n}-}^{\Qs^*}-M_{\tau^*_{n-1}}^{\Qs^*}\right)=0 \text{ for all }n \in \N,\ x \geq 0,$$ 
			\item 
			\[\lim_{T\rightarrow\infty}\frac{\E_{x}g(X^{\Qs^\ast}_T)}{T}=0 \text{  for all } x\in \R_+, \]
			
			\item  \[u(X^{\Qs^\ast}_{\tau_n^*-},y_0)=0  \ \ \Pro_{x} \text{-a.s. for all } x\in \R_+, \]
		\end{enumerate}
		then \[ v^h(x)= \rho, \text{ for all } x \in \R_+\]
		and $\Qs^*$ is optimal.
	\end{enumerate}
\end{lemma}

\begin{proof}
	We first fix an admissible $\Qs=(\tau_n)_{n\in\N}$ and $T>0$. Since the process $X^\Qs$ runs uncontrolled on each stochastic interval $[\tau_{k-1},\tau_k)$, the optional sampling theorem yields that 
 $\E_x[M^\Qs_{\tau_k\wedge T-}-M^\Qs_{\tau_{k-1}\wedge T}]\leq 0$ for each $k\in \N$, $x\in \R_+$. 
	Hence
	\begin{align*}
	& \E_x\left[\sum_{n \in \N :\tau_n\leq T}(f(X^\Qs_{\tau_n-})-{K}) -\int_0^T h(X^\Qs_s) \ ds \right]
	\\&\leq \E_x\left[\sum_{n \in \N:\tau_n\leq T} (f(X^\Qs_{\tau_n-})- K) -\sum_{k=1}^\infty \left(M^\Qs_{\tau_k\wedge T-}-M^\Qs_{\tau_{k-1}\wedge T}\right)-\int_0^T h(X^\Qs_s) \ ds\right]\\
	&=\E_x\left[\sum_{n \in \N:\tau_n\leq T}(f(X^\Qs_{\tau_n}) -K) -\sum_{k=1}^\infty \left(g(X^\Qs_{\tau_k\wedge T-})-g(X^\Qs_{\tau_{k-1}\wedge T}) \right.\right.\\& \left.\left.\textcolor{white}{-\sum_{k=1}^\infty}-\int_{\tau_{k-1}\wedge T}^{\tau_{k}\wedge T}(h(X^\Qs_s) +\rho) \ ds \right)-\int_0^T h(X^\Qs_s) \ ds\right] \\
	&=\E_x\left[\sum_{1\leq n:\tau_n\leq T}\left(f(X^\Qs_{\tau_n-})-K-g(X^\Qs_{\tau_n-})+g(y_0)\right)-g(X^\Qs_T)+g(X^\Qs_0)+\rho T\right]\\
	&=\E_x\left[\sum_{1\leq n:\tau_n\leq T}u(X^\Qs_{\tau_n-},y_0)\right]-\E_x \left[g(X^\Qs_T)\right] +g(x)+\rho T
	\\&\leq - \E_x g(X^\Qs_T)+g(x)+\rho T,
	\end{align*} where we use i.(c) and that $f(y_0)=0$ by assumption.
	
	Dividing by $T$ and taking the limit $T\rightarrow\infty$, we obtain the first assertion using i.(b) . The additional assumptions in \eqref{item:b} guarantee that we have equality in each step for the strategy ${\Qs}^*$. 
\end{proof}

We now provide a candidate and verify that this candidate satisfies the assumptions in Lemma \ref{veridrei}. 
The intuition for our candidate below is as follows: We first find one value $y$ from which shifting the process back to $y_0$ yields the maximal reward per time unit. Due to the continuity of $X$, we will always hit this point $y$ and therefore going on like this should yield an optimal strategy. 
More precisely, define
\begin{align*}
	\yy &:=\argmax_{y \in [y_0,\bar{y}]}\frac{f(y)-K-\E_{y_0}\left(\int_{0}^{\tau_y}h(X_s) \ ds\right)}{\xi(y)}, \\
	\rho^*&:=\max_{y \in [y_0,\bar{y}]}\frac{f(y)-K-\E_{y_0}\left(\int_{0}^{\tau_y}h(X_s) \ ds\right)}{\xi(y)},
\end{align*} and $Q^\ast = R(y^\ast)$. 
Moreover, set for all $x \in \mathbb{R}_+$ 
\begin{align*}
g(x):=\sup_{\tau \in \mathcal{T}_{y_0}} \E_x \left[f(X_\tau)-K -\int_0^\tau ( h(X_s)+\rho^*)  \ ds\right], \end{align*}
where $$\mathcal{T}_{y_0} := \{ \tau \text{ stopping time } | X_\tau \ge y_0 \text{ a.s.}\}.$$
Note that $g(x)$ is the value function of a classical stopping problem.

\begin{lemma}
	The function $g$, the constant $\rho^\ast$ and the strategy $R(\tau_\yy)$ fulfil the requirements in Lemma \ref{veridrei}.
\end{lemma}

\begin{proof}
	\textit{$M=({g(X_t)}-\int_0^t (h(X_s)+ \rho^\ast)   ds)_{t\geq0}$ is a supermartingale:} This is well-known by the standard theory of optimal stopping, but we give a direct proof here using the reverse optional sampling theorem. Let $\tau$ be a bounded stopping time. Then, using the time shift operator $\theta$,
	\begin{align*}
\E_x \left[M_\tau \right]&=	\E_x\left[{g(X_\tau)} -\int_0^\tau (h(X_s)+ \rho^\ast)   ds \right]
\\&=	\E_x \left[\sup_{\sigma \in \mathcal{T}_{y_0}}
\E_{X_\tau}\left[f(X_{\sigma})-K-\int_0^\sigma (h(X_s)+ \rho^\ast)   ds \right]-\int_0^\tau (h(X_s)+ \rho^\ast)   ds \right]
\\&=	\E_x \left[\sup_{\sigma \in \mathcal{T}_{y_0}}\E_{x}\left[f(X_{\sigma\circ \theta_\tau+\tau}) - K-\int_0^{\sigma\circ \theta_\tau+\tau} (h(X_s)+ \rho^\ast)   ds \mid \mathcal F_\tau\right] \right]
\\&\leq	\E_x \left[\sup_{\sigma \in \mathcal{T}_{y_0}; \sigma \geq \tau}\E_{x}\left[f(X_{\sigma})-K-\int_0^{\sigma} (h(X_s)+ \rho^\ast)   ds \mid \mathcal F_\tau\right]\right]
\\& \leq \sup_{\sigma \in \mathcal{T}_{y_0}}	\E_x \left[\E_{x}\left[f(X_{\sigma})-K-\int_0^{\sigma} (h(X_s)+ \rho^\ast)   ds \mid \mathcal F_\tau\right] \right]
\\& = g(x)
\\&= M_0
\end{align*}

\textit{The inequalities (i).c and (ii).c hold:} For this we investigate the stopping problem associated to $g$. 
	As a first step we note that by the general theory of optimal stopping, first entrance times into non-empty closed sets $S_\epsilon$ are $\epsilon$-optimal. (Note that problems with linear running costs are time homogenous, see, e.g., \cite{MR2083932}).
	 Since $X$ has continuous sample paths and due to the one-sided nature of the attainable stopping times, we have that\ $\Pro_{y_0}$-a.s. the first entrance into $S_\epsilon$ are identical to the first hitting times $\tau_{y_\epsilon}$, where $y_\epsilon=\min S_\epsilon$. Thus, we obtain
\begin{align*}
g(y_0)&=\sup_{y\geq y_0}\E_{y_0} \left[f(X_{\tau_y})-K -\int_0^{\tau_y} ( h(X_s)+\rho^*)  \ ds\right]\\
&=\sup_{y\geq y_0}\left[f(y)-K -\E_{y_0} \left[\int_0^{\tau_y}  h(X_s)  \ ds\right] -\max_{z \geq y_0}\frac{f(z)-K-\E_{y_0}\left[\int_{0}^{\tau_z}h(X_s) \ ds\right]}{\xi(z)}\xi(y)\right]\\
&=0.
\end{align*}
	Furthermore, we see that by construction the first hitting time of $y^\ast$ is optimal. In particular, this yields $g(y^\ast) = f(y^\ast) -K$.

Since immediate stopping is possible, we have $g \ge f-K$. This implies that
\begin{align*}
u(x,y_0) &= f(x)-f(y_0) - K - g(x) + g(y_0) = f(x)-K-g(x) \le 0,
\end{align*} which is (i).c. 
Analogously, we obtain, since $y^\ast$ is in the stopping region that $$u(X_{\tau_n^\ast-},y_0) = u(y^\ast, y_0) = 0.$$ 

\textit{(ii).a holds}:  The observation that $g(y^\ast) = f(y^\ast)-K$ directly yields that
\begin{align*}
	\mathbb{E}_x[ M_{\tau_n^\ast -} - M_{\tau_{n-1}^\ast}] &= \mathbb{E}_{y_0} \left[ g(X_{\tau_{y^\ast}-}) - \int_0^{\tau_{y^\ast}} (h(X_s) + \rho^\ast) ds -g (y_0) \right]\\&=
	\mathbb{E}_{y_0} \left[ g(y^\ast) - \int_0^{\tau_{y^\ast}} h(X_s)  ds -g (y_0) \right]- \rho^\ast\E_{y_0}\left[\tau_{y^\ast}\right]\\&= g(y^\ast) - (f(y^\ast) -K) = 0.
\end{align*}

	\textit{the transversality conditions (i).b and (ii).b hold:}
	As pointed out in Remark \ref{RemarkThresholdStrategies}, for any threshold strategy $R(y)$, $y \ge y_0$, the controlled process $X^{R(y)}$ has an invariant distribution.
	Denote the invariant measure for $Q^\ast = R(y^\ast)$ by $\Pi^{Q^\ast}$. 
	Thus, for all $y \in \mathbb{R}_+$ we obtain 
	$$ \lim_{T\rightarrow\infty} \E_{y}g \left[X^{\Qs^*}_T\right]=\int_0^{y^*}g(x)\pi^{\Qs^*}(dx).$$
Because $h$ is bounded on $[0,y^\ast]$ and $\tau_{y^\ast}$ is an optimal stopping time for the problem with value function $g$, we obtain for some $d>0$

 \begin{align*}
g(x)&=f(y^\ast)-K -\E_x \left[\int_0^{\tau_{y^\ast}} ( h(X_s)+\rho^*)  \ ds\right]\\
&\geq f(y^\ast)-K-d\E_x\tau_{y^\ast}\\
&=f(y^\ast)-K-d\left(\int_{x}^{y^\ast}(S(y^\ast)-S(y))m(y)\ dy+(S(y^\ast)-S(x))M[0,x]\right)\\
&\geq f(y^\ast)-K-d\int_{0}^{y^\ast}(S(y^\ast)-S(y))m(y)\ dy.
\end{align*}
As 0 was assumed to be an entrance-boundary, the last integral is finite and hence $g$ is bounded from below. Since $f$ is continuous, the function $g$ is furthermore bounded from above on compacts. Therefore, we obtain that  $$\lim\limits_{T\rightarrow \infty}\frac{\E_{x} [ g(X^{\Qs^*}_T)]}{T}=0 \text{ for all }x \in \R_+.$$

	Now take an arbitrary strategy ${\Qs}$ and denote with $X^r$ the process $X$ reflected at $y_0$.
	By the same coupling argument as in the proof of Lemma \ref{lem:bounds_z}, we can assume $X^r \leq X^ {\Qs}$.
	It is well known that in our setting $X^r$ has an invariant distribution $\Pi^r$. Further, as mentioned above, $g$ is bounded on $(0,y_0)$. 
	Utilizing that $f$ is increasing  and $h$ is a positive function a simple calculation yields that $g$ is increasing, which implies that $g(X^{\Qs})\geq g(X^r)$. 
	In total we obtain for all $y \in \mathbb{R}_+$ $$\liminf\limits_{T\rightarrow \infty}\E_{y}[g(X_T^Q)]\geq \liminf\limits_{T\rightarrow \infty}\E_{y}[g(X_T^r)] = \int g(x) \Pi^r(d x) \in \R$$ and therefore for all $y \in \R_+$ $$\limsup\limits_{T\rightarrow \infty}\frac{\E_{y}[g(X^\Qs_T)]}{T}\ge \liminf\limits_{T\rightarrow \infty}\frac{\E_{y}[g(X^r_T)]}{T}= 0.$$
\end{proof}
\begin{remark}
The renewal reward theorem directly yields that for each $y>y_0$ the value we get by using $R(y)$ equals 	$$\frac{f(y)-K-\E_{y_0}\left[\int_{0}^{\tau_y}h(X_s) \ ds\right]}{\xi(y)}.$$ Hence, whenever $\yy$ is the unique maximizer of $$y\mapsto \frac{f(y)-K-\E_{y_0}\left[\int_{0}^{\tau_y}h(X_s) \ ds\right]}{\xi(y)},  $$ then $R(\yy)$ is the unique optimal strategy in the set of threshold strategies. 
\end{remark}

\addcontentsline{toc}{chapter}{Bibliography}

\bibliography{literatur}

\def\cprime{$'$}
\begin{thebibliography}{10}

\bibitem{MR2023887}
{\sc L.~H.~R. Alvarez}, {\em On the properties of {$r$}-excessive mappings for
  a class of diffusions}, Ann. Appl. Probab., 13 (2003), pp.~1517--1533.

\bibitem{A04}
\leavevmode\vrule height 2pt depth -1.6pt width 23pt, {\em Stochastic forest
  stand value and optimal timber harvesting}, SIAM J. Control Optim., 42
  (2004), pp.~1972--1993.

\bibitem{alvarez2006does}
{\sc L.~H.~R. Alvarez and E.~Koskela}, {\em Does risk aversion accelerate
  optimal forest rotation under uncertainty?}, J. Forest Econ., 12 (2006),
  pp.~171--184.

\bibitem{andersson2011maximum}
{\sc D.~Andersson and B.~Djehiche}, {\em {A maximum principle for SDEs of
  mean-field type}}, Appl. Math. Optim., 63 (2011), pp.~341--356.

\bibitem{asmussen2008applied}
{\sc S.~Asmussen}, {\em Applied probability and queues}, vol.~51, Springer
  Science \& Business Media, 2008.

\bibitem{doi:10.1137/17M1119196}
{\sc A.~Aurell and B.~Djehiche}, {\em {Mean-Field Type Modeling of Nonlocal
  Crowd Aversion in Pedestrian Crowd Dynamics}}, SIAM J. Control Optim., 56
  (2018), pp.~434--455.

\bibitem{Belak2019}
{\sc C.~Belak and S.~Christensen}, {\em Utility maximisation in a factor model
  with constant and proportional transaction costs}, Finance Stoch., 23 (2019),
  pp.~29--96.

\bibitem{BensoussanMFG}
{\sc A.~Bensoussan, J.~Frehse, and P.~Yam}, {\em Mean Field Games and Mean
  Field Type Control Theory}, SpringerBriefs in Mathematics, Springer, New
  York, Heidelberg, Dordrecht, London, 2013.

\bibitem{BLi}
{\sc A.~Bensoussan and J.-L. Lions}, {\em Impulse control and quasivariational
  inequalities}, Gauthier-Villars, Montrouge, 1984.
\newblock Translated from the French by J. M. Cole.

\bibitem{bertucci2018optimal}
{\sc C.~Bertucci}, {\em Optimal stopping in mean field games, an obstacle
  problem approach}, Journal de Math{\'e}matiques Pures et Appliqu{\'e}es, 120
  (2018), pp.~165--194.

\bibitem{Bertucci}
{\sc C.~Bertucci}, {\em Fokker-planck equations of jumping particles and mean
  field games of impulse control}, Annales de l'Institut Henri Poincar{\'e} C,
  Analyse non lin{\'e}aire, 37 (2020), pp.~1211--1244.

\bibitem{borodin-salminen}
{\sc A.~Borodin and P.~Salminen}, {\em Handbook of {B}rownian Motion -- Facts
  and Formulae, 2nd edition, corrected printing}, Birkh\"auser, Basel, Boston,
  Berlin, 2015.

\bibitem{CainesHandbook}
{\sc P.~E. Caines, M.~Huang, and R.~P. Malham\'{e}}, {\em {Mean Field Games}},
  in Handbook of Dynamic Game Theory, T.~Basar and G.~Zaccour, eds., Springer,
  Cham, 2017.

\bibitem{CarmonaMFG2018}
{\sc R.~Carmona and F.~Delarue}, {\em Probabilistic Theory of Mean Field Games
  with Applications I: Mean Field FBSDEs, Control, and Games}, vol.~83 of
  Probability Theory and Stochastic Modelling, Springer International
  Publishing, 2018.

\bibitem{CarmonaMFGPartTwo2018}
\leavevmode\vrule height 2pt depth -1.6pt width 23pt, {\em Probabilistic Theory
  of Mean Field Games with Applications II: Mean Field Games with Common Noise
  and Master Equations}, vol.~84 of Probability Theory and Stochastic
  Modelling, Springer International Publishing, 2018.

\bibitem{Carmona2013}
{\sc R.~Carmona, F.~Delarue, and A.~Lachapelle}, {\em {Control of
  McKean--Vlasov dynamics versus mean field games}}, Math. Financ. Econ., 7
  (2013), pp.~131--166.

\bibitem{CarmonaBankRun}
{\sc R.~Carmona, F.~Delarue, and D.~Lacker}, {\em {Mean field games of timing
  and models for bank runs}}, Appl. Math. Optim., 76 (2017), pp.~217--260.

\bibitem{CecchinProbabilistic2018}
{\sc A.~Cecchin and M.~Fischer}, {\em {Probabilistic Approach to Finite State
  Mean Field Games}}, Applied Mathematics \& Optimization,  (2018).

\bibitem{C13impulse}
{\sc S.~Christensen}, {\em On the solution of general impulse control problems
  using superharmonic functions}, Stochastic Process. Appl., 124 (2014),
  pp.~709--729.

\bibitem{christensen2019solution}
{\sc S.~Christensen and T.~Sohr}, {\em A solution technique for l{\'e}vy driven
  long term average impulse control problems}, Stochastic Processes and their
  Applications, 130 (2020), pp.~7303--7337.

\bibitem{christensen2019nonparametric}
{\sc S.~Christensen and C.~Strauch}, {\em Nonparametric learning for impulse
  control problems}, arXiv preprint arXiv:1909.09528,  (2019).

\bibitem{clarke1989tree}
{\sc H.~R. Clarke and W.~J. Reed}, {\em {The tree-cutting problem in a
  stochastic environment: The case of age-dependent growth}}, J. Econ. Dyn.
  Control, 13 (1989), pp.~569--595.

\bibitem{EasthamHastings1988}
{\sc J.~F. Eastham and K.~J. Hastings}, {\em Optimal impulse control of
  portfolios}, Math. Oper. Res., 13 (1988), pp.~588--605.

\bibitem{MR3502394}
{\sc J.-S. Giet, P.~Vallois, and S.~Wantz-M\'{e}zi\`eres}, {\em The logistic
  {S}.{D}.{E}}, Theory Stoch. Process., 20 (2015), pp.~28--62.

\bibitem{gjolberg2002real}
{\sc O.~Gjolberg and A.~G. Guttormsen}, {\em Real options in the forest: what
  if prices are mean-reverting?}, Forest Policy and Econ., 4 (2002),
  pp.~13--20.

\bibitem{GomesConti2013}
{\sc D.~A. Gomes, J.~Mohr, and R.~R. Souza}, {\em {Continuous Time Finite State
  Mean Field Games}}, Appl. Math. Optim., 68 (2013), pp.~99--143.

\bibitem{ParisPrinceton2010}
{\sc O.~Gu\'{e}ant, J.-M. Lasry, and P.-L. Lions}, {\em {Mean Field Games and
  Applications}}, in Paris-Princeton Lectures on Mathematical Finance 2010,
  vol.~2003 of Lecture Notes in Mathematics, Springer-Verlag, Berlin,
  Heidelberg, 2011, pp.~205--266.

\bibitem{helmes2017continuous}
{\sc K.~L. Helmes, R.~H. Stockbridge, and C.~Zhu}, {\em Continuous inventory
  models of diffusion type: long-term average cost criterion}, Ann. Appl.
  Probab., 27 (2017), pp.~1831--1885.

\bibitem{helmes2017weak}
\leavevmode\vrule height 2pt depth -1.6pt width 23pt, {\em A weak convergence
  approach to inventory control using a long-term average criterion}, Adv. in
  Appl. Probab., 50 (2018), pp.~1032--1074.

\bibitem{HuangNCE2006}
{\sc M.~Huang, R.~P. Malham\'{e}, and P.~E. Caines}, {\em {Large population
  stochastic dynamic games: closed-loop McKean-Vlasov systems and the Nash
  certainty equivalence principle}}, Commun. Inf. Syst., 6 (2006),
  pp.~221--252.

\bibitem{MR2083932}
{\sc A.~Irle and V.~Paulsen}, {\em Solving problems of optimal stopping with
  linear costs of observations}, Sequential Anal., 23 (2004), pp.~297--316.

\bibitem{KarlinTaylor}
{\sc S.~Karlin and H.~M. Taylor}, {\em A second course in stochastic
  processes}, Academic Press, Inc., Montrouge, 1981.

\bibitem{korn1999}
{\sc R.~Korn}, {\em Some applications of impulse control in mathematical
  finance}, Math. Meth. Oper. Res., 50 (1999), pp.~493--518.

\bibitem{LasryJapanese2007}
{\sc J.-M. Lasry and P.-L. Lions}, {\em Mean field games}, Jp. J. Math., 2
  (2007), pp.~229--260.

\bibitem{MundacaOeksendal1998}
{\sc G.~Mundaca and B.~{\O}ksendal}, {\em Optimal stochastic intervention
  control with application to the exchange rate}, J. Math. Econom., 29 (1998),
  pp.~225--243.

\bibitem{NeumannComputation}
{\sc B.~A. Neumann}, {\em {Stationary Equilibria of Mean Field Games with
  Finite State and Action Space}}, Dynamic Games and Applications, 10 (2020),
  pp.~845--871.

\bibitem{NutzMFStopping}
{\sc M.~Nutz}, {\em {A Mean Field Game of Optimal Stopping}}, SIAM J. Control
  Optim., 56 (2018), pp.~1206--1221.

\bibitem{NutzCompetitionRD}
{\sc M.~Nutz and Y.~Zhang}, {\em {A Mean Field Competition}}, Math. Meth. Oper.
  Res.,  (2019).

\bibitem{OeksendalSulem2005}
{\sc B.~{\O}ksendal and A.~Sulem}, {\em Applied stochastic control of jump
  diffusions}, Springer, Berlin, 2005.

\bibitem{doi:10.1137/16M1085991}
{\sc J.~Palczewski and {\L}.~Stettner}, {\em {Impulse Control Maximizing
  Average Cost per Unit Time: A Nonuniformly Ergodic Case}}, SIAM J. Control
  Optim., 55 (2017), pp.~936--960.

\bibitem{Pretsch19}
{\sc H.~Pretsch}, {\em Grundlagen der Waldwachstumsforschung}, Springer
  Spektrum, Berlin, Heidelberg, 2019.

\bibitem{Richardsbaum59}
{\sc F.~J. Richards}, {\em A flexible growth function for empirical use},
  Journal of Experimental Botany,  (1959), pp.~290--300.

\bibitem{shackleton2010harvesting}
{\sc M.~B. Shackleton and S.~S{\o}dal}, {\em Harvesting and recovery decisions
  under uncertainty}, J. Econom. Dynam. Control, 34 (2010), pp.~2533--2546.

\bibitem{willassen1998stochastic}
{\sc Y.~Willassen}, {\em The stochastic rotation problem: a generalization of
  faustmann's formula to stochastic forest growth}, J. Econom. Dynam. Control,
  22 (1998), pp.~573--596.

\end{thebibliography}

\end{document}